\documentclass{article}
\usepackage[letterpaper,top=2cm,bottom=2cm,left=3cm,right=3cm,marginparwidth=1.75cm]{geometry}
\usepackage{natbib}
\usepackage{enumitem}
\setlist[description]{leftmargin=8.5em,labelwidth=10em,nosep}
\usepackage{amsmath,amsthm}
\usepackage{mathrsfs}
\usepackage{dsfont}
\usepackage{amsfonts}
\usepackage{amssymb}
\usepackage{bbm}
\usepackage{graphicx}
\usepackage[colorlinks=true, allcolors=blue]{hyperref}
\usepackage{natbib}
\usepackage{array}
\usepackage{lscape}
\usepackage{xspace}
\usepackage{bm}
\usepackage{booktabs}

\newcommand{\bX}{\bm{X}}

\newcommand{\mw}{\bm{w}}
\newcommand{\bx}{\bm{x}}
\newcommand{\by}{\bm{y}}
\newcommand{\R}{\mathbb{R}}
\newcommand{\bZ}{\bm{Z}}
\newcommand{\bz}{\bm{z}}

\makeatletter
\DeclareRobustCommand\widecheck[1]{{\mathpalette\@widecheck{#1}}}
\def\@widecheck#1#2{%
    \setbox\z@\hbox{\m@th$#1#2$}%
    \setbox\tw@\hbox{\m@th$#1%
       \widehat{%
          \vrule\@width\z@\@height\ht\z@
          \vrule\@height\z@\@width\wd\z@}$}%
    \dp\tw@-\ht\z@
    \@tempdima\ht\z@ \advance\@tempdima2\ht\tw@ \divide\@tempdima\thr@@
    \setbox\tw@\hbox{%
       \raise\@tempdima\hbox{\scalebox{1}[-1]{\lower\@tempdima\box
\tw@}}}%
    {\ooalign{\box\tw@ \cr \box\z@}}}
\makeatother

\newcommand\Var{\mathrm{Var}}
\newcommand\E{\mathrm{E}}
\newcommand\ds{\mathrm{d}}
\newcommand\m{\mathcal{M}}
\newcommand\graph{\mathcal{G}_n}

\newcommand\X{\bX}

\newcommand{\mres}{\mathbin{\vrule height 1.6ex depth 0pt width
0.13ex\vrule height 0.13ex depth 0pt width 1.3ex}}

\newcommand{\ko}{\mathfrak{o}}

\newcommand{\kq}{\mathfrak{q}}

\newtheorem{lemma}{\bf Lemma}[section]
\newtheorem{assumption}{\bf Assumption}[section]
\newtheorem{theorem}{\bf Theorem}[section]
\newtheorem{definition}{\bf Definition}[section]
\newtheorem{remark}{Remark}[section]

\numberwithin{equation}{section}

\allowdisplaybreaks

\begin{document}

\setlength{\abovedisplayskip}{5pt}
\setlength{\belowdisplayskip}{5pt}
\setlength{\abovedisplayshortskip}{5pt}
\setlength{\belowdisplayshortskip}{5pt}
\hypersetup{colorlinks,breaklinks,urlcolor=blue,linkcolor=blue}

\title{Azadkia-Chatterjee's correlation coefficient adapts to manifold data}

\author{
Fang Han\thanks{Department of Statistics, University of Washington, Seattle, WA 98195, USA; e-mail: {\tt fanghan@uw.edu}} ~~and~ Zhihan Huang\thanks{School of Mathematical Science, Peking University, Beijing, China; e-mail: {\tt 1900010737@pku.edu.cn}}
}

\date{}

\maketitle

\vspace{-1em}

\begin{abstract} 
In their seminal work, \cite{azadkia2019simple} initiated graph-based methods for measuring variable dependence strength. By appealing to nearest neighbor graphs, they gave an elegant solution to a problem of R\'enyi \citep{renyi1959measures}. Their idea was later developed in \cite{deb2020kernel} and the authors there proved that, quite interestingly, Azadkia and Chatterjee's correlation coefficient can automatically adapt to the manifold structure of the data. This paper furthers their study in terms of calculating the statistic's limiting variance under independence and showing that it only depends on the manifold dimension.
\end{abstract}


\section{Introduction}

Consider $\bX \in  \mathbb{R}^d$ and $Y\in\mathbb{R}$ to be two random variables defined over the same probability space with fixed and continuous joint distribution function $F_{\bX,Y}$ and marginal distributions $F_{\bX}$ and $F_Y$, respectively. Let $(\bX_1,Y_1),\ldots,(\bX_n,Y_n)$ be $n$ independent copies of $(\bX,Y)$, $R_i$ be the rank of $Y_i$ among $\{Y_1, \ldots, Y_n\}$, and $N(i)$ index the nearest neighbor (NN) of $\bX_i$ among $\{\bX_1, \ldots, \bX_n\}$. Built on an earlier work of \cite{chatterjee2020new}, \cite{azadkia2019simple}  introduced  the following graph-based correlation coefficient
\[
\xi_n =\xi_n(\{(\bX_i,Y_i)\}_{i=1}^n):= \frac{6}{n^2-1} \sum_{i=1}^n \min\Big\{R_i,R_{N(i)}\Big\}-\frac{2n+1}{n-1}.
\]
This correlation coefficient was shown in \citet[Theorem 2.2]{azadkia2019simple} to converge strongly to a population dependence measure that was first introduced in \cite{dette2013copula},
\begin{align*}
    \xi=\xi(\X,Y):=\frac{\displaystyle\int\Var\big\{\E\big[\mathds{1}\big(Y\geq t\big)|\X\big]\big\}\ds F_{Y}(t)}{\displaystyle\int\Var\big\{\mathds{1}\big(Y\geq t\big)\big\}\ds F_{Y}(t)}.
\end{align*}
\cite{dette2013copula}'s dependence measure satisfies some of the most desirable properties discussed in \cite{renyi1959measures} including, in particular, the following three:
\begin{itemize}
\item[(1)] $\xi\in [0,1]$;
\item[(2)] $\xi=0$ if and only if $\bX$ is independent of $Y$;
\item[(3)] $\xi=1$ if and only if $Y$ is a measurable function of $\bX$ almost surely. 
\end{itemize}
Azadkia and Chatterjee thus outlined an elegant approach to measuring the dependence strength between $\bX$ and $Y$, resolving many long-standing issues that surround R\'enyi's criteria as were recently discussed by Professor Peter Bickel \citep{bickel2022measures}.  

The authors of this paper are interested in $\xi_n$'s adaptivity to the manifold structure of the data, a problem that has received much interest in the NN literature \citep{levina2004maximum,kpotufe2011k,kpotufe2013adaptivity}. To this end, our focus is on the limiting null distribution of $\xi_n$, i.e., its limiting distribution under independence between $\bX$ and $Y$. Such a result, if derived, would immediately give rise to a statistical test of the following null hypothesis, 
\[
H_0:~~\bX \text{ (supported on a smooth manifold) is independent of }Y.
\]

Below is the main result of this paper.

\begin{theorem}[Central limit theorem of $\xi_n$ for manifold data]\label{thm:main}
Let $Y \in \mathbb{R}$ be independent of $\X\in\mathbb{R}^d$ and let $F_{\bX,Y}$ be fixed and continuous. Further assume that the following two conditions hold: 
\begin{enumerate}[itemsep=-.5ex,label=(\roman*)]
\item\label{assump:manifold} $\X\in \m$, where $\m$ is an $m$-dimensional $C^{\infty}$ manifold in $\mathbb{R}^d$ with $m\leq d$;
\item\label{assump:measure} the law of $\X$ is absolutely continuous with respect to $\mathcal{H}^m\mres\m$, the restricted $m$-dimensional Hausdorff measure in $\mathbb{R}^d$ on $\m$.
\end{enumerate}
We then have, as $n\rightarrow \infty$,
    \begin{align*}
        \sqrt{n}\xi_n \text{ converges to } N\Big(0,\frac{2}{5}+\frac{2}{5}\mathfrak{q}_{m}+\frac{4}{5}\mathfrak{o}_{m}\Big) \text{ in distribution},
    \end{align*}
    where for any integer $m\geq 1$,
    \begin{align*}
&\kq_m:=\Big\{2-I_{3/4}\Big(\frac{m+1}{2},\frac{1}{2}\Big)\Big\}^{-1},
~~~~
I_{x}(a,b):=\frac{\int_{0}^{x}t^{a-1}(1-t)^{b-1} \ds t}
                 {\int_{0}^{1}t^{a-1}(1-t)^{b-1} \ds t},\\
&\ko_{m}:=\iint_{\Gamma_{m;2}}\exp\Big[-\lambda\Big\{B(\mw_1,\lVert\mw_1\rVert_{})\cup B(\mw_2,\lVert\mw_2\rVert_{})\Big\}\Big]\ds(\mw_1,\mw_2),\\
&\Gamma_{m;2}:=\Big\{(\mw_1,\mw_2)\in(\R^m)^2: \max(\lVert\mw_1\rVert_{},\lVert\mw_2\rVert_{})<\lVert\mw_1-\mw_2\rVert_{}\Big\},
\end{align*} 
$B(\mw_1,r)$ represents the ball of radius $r$ and center $\mw_1$, $\|\cdot\|$ is the Euclidean norm, and $\lambda(\cdot)$ is the Lebesgue measure. 
\end{theorem}

\begin{remark}
In Theorem \ref{thm:main}, we assume a constant global dimension of $\m$. When the data structure is more complex,  the dimension may differ between different (connected or not) components of $\m$. In such cases, the value of $\xi_n$'s asymptotic variance can be derived analogously as a mixture distribution of each part corresponding to one component of $\m$. 
\end{remark}

As we will explain later in Section \ref{sec:sketch}, the terms $\kq_m$ and $\ko_m$ count the averaged numbers of nearest neighbor pairs and triples, respectively. The first ten $\kq_m$ and $\ko_m$ were shown in Table \ref{tab:kq} and some basic properties are listed below.

\begin{lemma}\label{cons}
The following holds true:
\begin{itemize}
\item[(a)] $\kq_m\in\big(\frac{1}{2},\frac{2}{3}\big]$ is strictly decreasing as $m$ increases;
\item[(b)] $\sup_{m}\ko_m<2$ and $\limsup_m\ko_m\leq 1$.
\end{itemize}
\end{lemma}

\begin{table}[t]
\centering
\setlength{\tabcolsep}{3mm}{ 
\caption{The first 10 values of $\kq_m$ and $\ko_m$.}
\label{tab:kq}
\begin{tabular}{ccccccccccc}
\hline
$m$ & 1 & 2 & 3 & 4 & 5 &6 &7 & 8 & 9 & 10 \\ \hline
$\mathfrak{q}_m$  & 0.67& 0.62 & 0.59 & 0.57 & 0.56 & 0.55 & 0.54 & 0.53 & 0.53 & 0.52 \\
$\mathfrak{o}_m$  & 0.49 & 0.63 & 0.71 & 0.76 & 0.79 & 0.84 & 0.86 & 0.90 & 0.98 & 1.00 \\ \hline
\end{tabular}}
\end{table}

\subsection{Literature review}

Theorem \ref{thm:main} builds bridge between two statistical fields, the study of graph-based correlation coefficients and the study of nearest neighbor methods and their adaptivity to manifold data. 

On one hand, since the pioneering work of  \cite{chatterjee2020new} and \cite{azadkia2019simple}, the study of graph-based correlation coefficients has quickly attracted attention; a literature is being built up rapidly and includes, among many others, \cite{cao2020correlations}, \cite{shi2020power}, \cite{gamboa2020global}, \cite{deb2020kernel}, \cite{huang2020kernel}, \cite{auddy2021exact}, \cite{shi2021ac}, \cite{lin2021boosting}, \cite{fuchs2021bivariate}, \cite{azadkia2021fast}, \cite{griessenberger2022multivariate}, \cite{strothmann2022rearranged}, \cite{lin2022limit}, \cite{zhang2022asymptotic}, \cite{bickel2022measures}, and \cite{chatterjee2022estimating}. 

In the following we list three existing results that are most relevant to Theorem \ref{thm:main}. Readers of more interest are referred to \cite{han2021extensions} and \citet[Section 1.1]{lin2022limit} for a slightly more complete review. 
\begin{enumerate}[itemsep=-.5ex,label=(\arabic*)]
\item \cite{deb2020kernel} studied a general class of graph-based correlation coefficients, to which $\xi_n$ belongs to. Their Corollary 5.1 examined the {\it convergence rate} for $\xi_n$ to $\xi$, illustrating an interesting interplay between the intrinsic dimension of $\bX$ and the smoothness of some conditional expectation functions relating $Y$ to $\bX$. They revealed, for the first time, the adaptation of graph-based correlation coefficients to the manifold structure of $\bX$. 
\item Built on the work of \cite{deb2020kernel}, \citet[Theorem 3.1(ii)]{shi2021ac} established a {\it central limit theorem (CLT)} for $\xi_n$ under (a) independence between $Y$ and $\bX$; (b) {\it absolute continuity} (with respect to the Lebesgue measure) of $F_{\bX}$. Under these conditions, they showed
\[
        \sqrt{n}\xi_n \text{ converges to } N\Big(0,\frac{2}{5}+\frac{2}{5}\mathfrak{q}_{d}+\frac{4}{5}\mathfrak{o}_{d}\Big) \text{ in distribution}.
\]
\item In a more recent preprint, \citet[Theorem 1.1]{lin2022limit} established a CLT for $\xi_n$ while removing both independence and absolute continuity assumptions required in \cite{shi2021ac}. In particular, they showed that as long as (a) $F_{\bX, Y}$ is fixed and continuous and (b) $Y$ is not almost surely a measurable function of $\bX$, it holds true that
\[
(\xi_n-\E\xi_n)/\sqrt{\Var(\xi_n)} \text{ converges to } N(0,1) \text{ in distribution}.
\]
\end{enumerate}
Theorem \ref{thm:main} can thus be viewed as an descendent of the above three results: 
\begin{enumerate}[itemsep=-.5ex,label=(\arabic*)]
\item compared to \cite{deb2020kernel}, it established a weak convergence instead of a point estimation type result; 
\item compared to \cite{shi2021ac}, it removed the absolute continuity assumption required therein; 
\item compared to \citet[Theorem 1.1]{lin2022limit}, Theorem \ref{thm:main} calculated the explicit value of the asymptotic variance.  
\end{enumerate}

On the other hand, in practice many data are believed to be structured, i.e., they are embedded in a space that is of a much higher dimension than necessary \citep{levina2004maximum,amelunxen2014living}. Local methods, especially the NN-based ones, are long time believed to be suitable for analyzing such data, capably of automatically adapting to the data structure \citep{clarkson2006nearest,kpotufe2011k,kpotufe2013adaptivity,kpotufe2017lipschitz}. We believe Theorem \ref{thm:main} also bears potential to contribute to this line of the research. In particular, 
\begin{enumerate}[itemsep=-.5ex,label=(\arabic*)]
\item as we shall show in Section \ref{sec:sketch}, an essence of Theorem \ref{thm:main} is to calculate the 
averaged numbers of nearest neighbor pairs and triples; they are thus monitoring the stochastic structure of an NN graph (NNG) when the data are distributed over a manifold.
\item Theorem \ref{thm:main} is also, to our knowledge, the {\it first} weak convergence type results for tracking the statistical behavior of a NN-based functional over a manifold-supported probability space.
\end{enumerate}

\subsection{Proof sketch}\label{sec:sketch}

We first introduce some auxiliary results on the NNGs and the manifold. 
Recall that $[\textsl{\textbf{X}}_i]_{i=1}^n$ comprise $n$ independent copies of a random vector $\textsl{\textbf{X}} \in \mathbb{R}^d$ from an unknown distribution function $F_{\bX}$. Let $\graph$ be the associated directed NNG with vertex set $\{1,...,n\}$ and edge set $\mathcal{E}(\graph)$; here an edge $\{i \rightarrow j\}\in \mathcal{E}(\graph)$ means $\bX_j$ is the NN of $\bX_i$. 

We are interested in manifold data; more precisely, we are interested in such random vector $\textsl{\textbf{X}}$ that is supported on $\m$, a smooth submanifold of $\mathbb{R}^d$ with manifold dimension $m \leq d$. The following concepts are from \cite{lee2013smooth}.

\begin{definition}
Let $\m$ be an $m$-dimensional smooth manifold. A {\it coordinate chart}, abbreviated as a {\it chart}, on $\m$ is a pair $(U,\psi)$, where $U$ is an open set of $\m$ and $\psi:U\rightarrow V$ is a homeomorphism from $U$ to an open subset $V=\psi(U)\subset \mathbb{R}^m$.
\end{definition}

\begin{definition}
    Given a smooth manifold $\m$ and a chart $(U,\psi)$ of $\m$, $U$ is called a {\it coordinate neighborhood} of each point $w \in U$.
\end{definition}

Concerning any point $\bx\in\m$, one can find a chart of $\m$ with coordinate neighborhood $U(=U_{\bx})$ and corresponding homeomorphism $\psi(=\psi_{\bx})$.\footnote{There  are, of course, many such neighborhoods and homeomorphisms circling $\bx$; in the sequel we simply pick one of them.} With this notion, we give an assumption on the distribution of $\textsl{\textbf{X}}$ that will be shown to be an alternative to Theorem \ref{thm:main}\ref{assump:measure}. In the following, the law of $\textsl{\textbf{X}}$ is denoted by $\mu$ and the restriction of $\mu$ to a set $U$ is denoted by $\mu_{U}$.

\begin{assumption}\label{measure assump}
    \emph{(Distribution assumption).} The positive measure $\mu$ satisfies the following condition: for any point $\bx\in\m$ and any chart $(U,\psi)$ such that 
    $U$ is a coordinate neighborhood of $\bx$ and $\psi:U\rightarrow V \subset \mathbb{R}^m$, the restricted pushforward measure $\psi_{*}\mu_{U}$ is absolutely continuous with respect to the Lebesgue measure $\lambda(\cdot)$ on $V$. 
\end{assumption}

Assumption \ref{measure assump} yields an alternate description of the data generating process and will appear to be useful in the following proofs; 
the next lemma shows that it is equivalent to the assumption of Theorem \ref{thm:main}\ref{assump:measure}.

\begin{lemma}\label{al measure assumption}
    \emph{(Alternative distribution assumption).} As $\m$ satisfies the assumption of Theorem \ref{thm:main}\ref{assump:manifold}, Assumption \ref{measure assump} is equivalent to the assumption of Theorem \ref{thm:main}\ref{assump:measure}.
\end{lemma}

We then move on to study the stochastic behavior of the NNG $\mathcal{G}_n$ as $\m$ satisfies Theorem \ref{thm:main}\ref{assump:measure} and $\bX$ satisfies Assumption \ref{measure assump}. Since in the expression of $\xi_n$, one of the rank terms is indexed by its NNG, the asymptotic distribution of $\xi_n$ has a connection with the properties of the NNG. In the following we introduce a series of lemmas on this topic. To begin with, Lemma \ref{max degree} is a well-known result by \cite{bickel1983sums} on the 
maximum number of nearest neighbors.

\begin{lemma}\label{max degree}
    \emph{(Maximum degree in nearest neighbor graphs).} There is an upper bound for the degree of any point in NNGs. More specifically, let $\bx_1,...,\bx_n$ be any collection of $n$ distinct points in $\mathbb{R}^d$. Then there exists a constant $\mathfrak{C}_d$ depending only on the dimension $d$ such that $\bx_1$ is the nearest neighbor of at most $\mathfrak{C}_d$ points from $\{\bx_2,\ldots, \bx_n\}$.
\end{lemma}


The next two lemmas draw the average numbers of some specific structures in an NNG. They are extensions of conclusions in \cite{devroye1988expected} and \cite{henze1987fraction}. We first focus on the number of loops between two vertices, which we call a {\it nearest neighbor pair} in $\mathcal{G}_n$.
\begin{lemma}\label{q}
    \emph{(Expected number of nearest-neighbor pairs).} Consider $\graph$ in $\mathbb{R}^d$ with $\m$ and $\mu$ satisfying the assumption of Theorem \ref{thm:main}\ref{assump:manifold} and Assumption \ref{measure assump}, respectively. We then have, as $n\rightarrow \infty$,
    \begin{align*}
        \E\bigg(\frac{1}{n}\#\Big\{(i,j) \enspace distinct: i\rightarrow j, j\rightarrow i \in \mathcal{E}(\graph)\Big\}\bigg)\rightarrow\frac{V_m}{U_m}:=\kq_m,
    \end{align*}
    where $V_m$ is the volume of the unit ball in $\mathbb{R}^m$, and $U_m$ is the volume of the union of two unit balls in $\mathbb{R}^m$ whose centers are a unit distance apart.
\end{lemma}

We then turn to another structure in $\mathcal{E}(\mathcal{G}_n)$ that monitors those parent vertices that share the same child vertex. We call them a {\it nearest neighbor triple} in $\mathcal{G}_n$.
\begin{lemma}\label{o}
 Consider $\graph$ in $\mathbb{R}^d$ with $\m$ and $\mu$ satisfying the assumption of Theorem \ref{thm:main}\ref{assump:manifold} and Assumption \ref{measure assump}, respectively. We then have, as $n\rightarrow \infty$,
    \begin{align*}
        \mathrm{E}\bigg(\frac{1}{n}\#\Big\{(i,j,k)\enspace distinct: i\rightarrow k,j\rightarrow k\in \mathcal{E}(\graph)\Big\}\bigg)\rightarrow\ko_m.
    \end{align*}
\end{lemma}

Get back to the data $(\bX_i,Y_i), \enspace i=1,...,n$. Let assumptions in Theorem \ref{thm:main} hold and we construct $\graph$ on the manifold data $[\bX_i]_{i=1}^n$. Recall that $m$ denotes the manifold dimension and $\mathfrak{q}_{m}$, $\mathfrak{o}_{m}$ are positive constants depending only on $m$. With the lemmas presented above, it is then straightforward to derive the limiting variance following the proof of \citet[Theorem 3.1]{shi2021ac}.
\begin{theorem}\label{main2}
Suppose that the conditions in Theorem \ref{thm:main} hold. We then have, as $n\rightarrow \infty$,
\begin{align*}
        \Var\Big(\sqrt{n}\xi_n\Big) \text{ converges to } \frac{2}{5}+\frac{2}{5}\kq_{m}+\frac{4}{5}\ko_{m} \text{, as }n\to\infty.
    \end{align*}
\end{theorem}

Lastly, when $Y$ is independent of $\X$, using Theorem 1.1 in \cite{lin2022limit}, we have
\begin{equation*}
    \frac{\xi_n}{\sqrt{\Var(\xi_n)}} \stackrel{d}{\rightarrow} N(0,1).
\end{equation*}
Combining the above result with Theorem \ref{main2} then yields Theorem \ref{thm:main}.

\subsection{Some finite-sample studies}

This section contains some finite-sample simulation results to examine the independence test powers, comparing the performance of $\xi_n$  to that of distance correlation \citep{MR2382665}. We examine the sizes and powers of the proposed tests when the data are supported on a manifold satisfying theorem assumptions and with manifold dimension {\it known to us}.\footnote{In practice, if the manifold dimension is unknown, one could either estimate it or use permutation to obtain the test threshold.} Power comparisons are carried out with sample size $n = 100$. In each case, 5,000 simulations are used to calculate the empirical size/power.

We first generate the raw data $(Y_i,\bZ_i)$, $i=1,...,n$. Here $(Y_1,\bZ_1),...,(Y_n,\bZ_n)$ constitute an sample of points independently drawn from a certain distribution on $\mathbb{R}\times\mathbb{R}^m$. The value of $m$ will change in simulations.

\begin{itemize}
\item Case 1 (Gaussian): $(Y,\bZ)$ is Gaussian distributed with mean $\bm{0}$ and equi-correlation $\rho$ between $Y$ and each component of $\bZ$, i.e., $(Y,\bZ)\sim N(\bm{0},\Sigma)$ with
\[
\Sigma :=
\begin{pmatrix}
    1 & \rho \bm{1}_m^{\top}\\
    \rho \bm{1}_m & I_m
\end{pmatrix},
\]
where $\bm{1}_m:=(\underbrace{1,\ldots,1}_{m})^\top$ and $I_m$ represents the $m$-dimensional identity matrix.
\end{itemize}

In the following five cases, we set $\bZ = (Z_1,...,Z_m)^\top$, where $Z_j\sim \mathrm{Unif}[-1,1]$ is independent of each other, and consider an additive model:
\begin{align*}
    Y = \rho \sum_{i=1}^m f(Z_i)+C\epsilon,
\end{align*}
where $\epsilon \sim N(0,1)$ is independent of $\bZ$. We fix $C$ as a constant in each case to modify the noise intensity.

\begin{itemize}
\item Case 2 (Linear): $f(x) = x$, $C=0.2$;
\item Case 3 (Quadratic): $f(x) = x^2$, $C=0.1$;
\item Case 4 (Cosine): $f(x) = \cos(8\pi x)$, $C=0.1$;
\item Case 5 (W-shape): $f(x) = \vert x+0.5\vert I_{\{x<0\}}+\vert x-0.5\vert I_{\{x\geq0\}}$, $C=0.025$.
\end{itemize}

Regarding each of the five cases, we then conduct the following two types of transformation to obtain the manifold data. 
\begin{enumerate}[itemsep=-.5ex,label=(\arabic*)]
\item Linear transformation: $\bZ \mapsto R\bZ =: \bX$, where $R$ is a pre-selected $5m$ by $m$ matrix. For each dimension $m$,  we randomly generate $R$ from a standard Gaussian random matrix $(R_{i j})_{5m\times m}$, where all its elements are independent, and for $i=1,...,5m$, $j=1,...,m$, $R_{i j}\sim N(0,1)$. 
The sample of transformed points lies on a $m$-dimensional linear subspace in $\mathbb{R}^{5m}$.
\item Manifold transformation: $\bZ \mapsto M(\bZ)=:\bX$, which is a map from $\R^m$ to a pre-specific $m$-dimensional smooth manifold in $\R^{5m}$. In the following simulations, the mapping takes the following specific forms:
\begin{align*}
    M(\bZ) = (M_1(\bZ),M_2(\bZ),M_3(\bZ),M_4(\bZ),M_5(\bZ))^\top,
\end{align*}
where
\begin{align*}
    M_1(\bZ):=&(Z_1,...,Z_m)=\bZ,\\
    M_2(\bZ):=&(Z_{1}^2,...,Z_{m}^2),\\
    M_3(\bZ):=&(\sin(8\pi Z_{1}),...,\sin(8\pi Z_{m})),\\
    M_4(\bZ):=&(\cos(4\pi Z_{1}),...,\cos(4\pi Z_{m})),\\
    M_5(\bZ):=&(\exp(Z_{1}),...,\exp(Z_{m})).
\end{align*}
\end{enumerate}

We perform tests of independence with parameters $m=1,2,3,5,10$ and $\rho = 0,0.05,0.10,0.15,0.20$ for both $\xi_n$ and distance correlation. Nominal level is set to be $\alpha=0.05$ for all tests. In all the cases, $\rho = 0$ corresponds to the null hypothesis 
\[
H_0: Y \text{ and }\bX  \text{ are independent,} 
\]
while the rest values of $\rho$ yield powers in accordance with different degrees of dependence. The thresholds of $\xi_n$ and distance correlation are determined by Theorem \ref{thm:main} and permutation, respectively. Table \ref{tab:1.1} and Table \ref{tab:1.2} illustrate test powers for Gaussian (Case 1) with linear and manifold transformation, respectively. Tables \ref{tab:2.1}-\ref{tab:6.2} analogously illustrate test powers for the additive model cases (Case 2-6) with two transformations in sequence.

Three observations are in line. 
\begin{enumerate}[itemsep=-.5ex,label=(\roman*)]
\item All the tests considered have empirical sizes close to 0.05, indicating that they are all size valid. 
\item When the joint distribution of $(Y,\bZ)$ is a multi-dimension normal distribution, the test power increases as $m$ or $\rho$ increases. The distance correlation based tests exhibit higher power compared to $\xi_n$-based, indicating that the distance correlation could potentially also adapt to the manifold structure of $\bX$, an interesting phenomenon largely untouched in literature before. 
\item When the function $f$ exhibits oscillatory properties, the test power increases as $\rho$ increases. However, when $m$ increases, in most cases, the power of our proposed tests shows a U-shape, i.e., decreasing first and then increasing; yet the distance correlation based tests shows monotonically decreasing power. Thus, for high-dimensional data, our proposed test might be more powerful. When $m=1$, our proposed test is also more powerful in some cases. 
\end{enumerate}

\begin{table}[!ht]
\centering
\setlength{\tabcolsep}{1.2mm}{ 
\caption{Case 1, linear transformation}
\label{tab:1.1}
\begin{tabular}{ccccccccccc}
\toprule
\multicolumn{1}{c}{} & \multicolumn{5}{c}{$\xi_n$ based}                                                                                                                                             & \multicolumn{5}{c}{distance correlation based} \\                                                                                                                                   \cmidrule(lr){2-6}\cmidrule(lr){7-11}
                       & \multicolumn{1}{c}{$\rho = 0$} & \multicolumn{1}{c}{$\rho = 0.05$} & \multicolumn{1}{c}{$\rho = 0.10$} & \multicolumn{1}{c}{$\rho = 0.15$} & \multicolumn{1}{c}{$\rho=0.20$} & \multicolumn{1}{c}{$\rho = 0$} & \multicolumn{1}{c}{$\rho = 0.05$} & \multicolumn{1}{c}{$\rho = 0.10$} & \multicolumn{1}{c}{$\rho = 0.15$} & \multicolumn{1}{c}{$\rho=0.20$} \\ 
$m = 1$               & \multicolumn{1}{c}{0.050}      & \multicolumn{1}{c}{0.051}         & \multicolumn{1}{c}{0.048}         & \multicolumn{1}{c}{0.053}         & 0.054                            & \multicolumn{1}{c}{0.053}      & \multicolumn{1}{c}{0.066}         & \multicolumn{1}{c}{0.144}         & \multicolumn{1}{c}{0.271}         & 0.455                            \\ 
$m = 2$               & \multicolumn{1}{c}{0.044}      & \multicolumn{1}{c}{0.048}         & \multicolumn{1}{c}{0.042}         & \multicolumn{1}{c}{0.046}         & 0.052                            & \multicolumn{1}{c}{0.044}      & \multicolumn{1}{c}{0.080}         & \multicolumn{1}{c}{0.208}         & \multicolumn{1}{c}{0.405}         & 0.609                            \\ 
$m = 3$               & \multicolumn{1}{c}{0.050}      & \multicolumn{1}{c}{0.049}         & \multicolumn{1}{c}{0.055}         & \multicolumn{1}{c}{0.056}         & 0.079                            & \multicolumn{1}{c}{0.052}      & \multicolumn{1}{c}{0.074}         & \multicolumn{1}{c}{0.229}         & \multicolumn{1}{c}{0.454}         & 0.766                            \\ 
$m = 5$               & \multicolumn{1}{c}{0.057}      & \multicolumn{1}{c}{0.058}         & \multicolumn{1}{c}{0.049}         & \multicolumn{1}{c}{0.065}         & 0.112                            & \multicolumn{1}{c}{0.041}      & \multicolumn{1}{c}{0.141}         & \multicolumn{1}{c}{0.331}         & \multicolumn{1}{c}{0.604}         & 0.844                            \\ 
$m = 10$              & \multicolumn{1}{c}{0.066}      & \multicolumn{1}{c}{0.060}         & \multicolumn{1}{c}{0.057}         & \multicolumn{1}{c}{0.099}         & 0.215                            & \multicolumn{1}{c}{0.043}      & \multicolumn{1}{c}{0.106}         & \multicolumn{1}{c}{0.420}         & \multicolumn{1}{c}{0.842}         & 0.995                            \\ 
\bottomrule
\end{tabular}}
\end{table}

\begin{table}[!ht]
\centering
\setlength{\tabcolsep}{1.2mm}{ 
\caption{Case 1, manifold transformation}
\label{tab:1.2}
\begin{tabular}{ccccccccccc}
\toprule
\multicolumn{1}{c}{} & \multicolumn{5}{c}{$\xi_n$ based}                                                                                                                                             & \multicolumn{5}{c}{distance correlation based} \\                                                                                                                                   \cmidrule(lr){2-6}\cmidrule(lr){7-11}
                       & \multicolumn{1}{c}{$\rho = 0$} & \multicolumn{1}{c}{$\rho = 0.05$} & \multicolumn{1}{c}{$\rho = 0.10$} & \multicolumn{1}{c}{$\rho = 0.15$} & \multicolumn{1}{c}{$\rho=0.20$} & \multicolumn{1}{c}{$\rho = 0$} & \multicolumn{1}{c}{$\rho = 0.05$} & \multicolumn{1}{c}{$\rho = 0.10$} & \multicolumn{1}{c}{$\rho = 0.15$} & \multicolumn{1}{c}{$\rho=0.20$} \\ \hline
$m = 1$               & \multicolumn{1}{c}{0.058}      & \multicolumn{1}{c}{0.051}         & \multicolumn{1}{c}{0.051}         & \multicolumn{1}{c}{0.050}         & 0.059                            & \multicolumn{1}{c}{0.041}      & \multicolumn{1}{c}{0.051}         & \multicolumn{1}{c}{0.114}         & \multicolumn{1}{c}{0.269}         & 0.336                            \\ 
$m = 2$               & \multicolumn{1}{c}{0.058}      & \multicolumn{1}{c}{0.055}         & \multicolumn{1}{c}{0.059}         & \multicolumn{1}{c}{0.058}         & 0.065                            & \multicolumn{1}{c}{0.048}      & \multicolumn{1}{c}{0.088}         & \multicolumn{1}{c}{0.122}         & \multicolumn{1}{c}{0.287}         & 0.548                            \\  
$m = 3$               & \multicolumn{1}{c}{0.075}      & \multicolumn{1}{c}{0.068}         & \multicolumn{1}{c}{0.065}         & \multicolumn{1}{c}{0.071}         & 0.079                            & \multicolumn{1}{c}{0.060}      & \multicolumn{1}{c}{0.056}         & \multicolumn{1}{c}{0.106}         & \multicolumn{1}{c}{0.448}         & 0.652                            \\  
$m = 5$               & \multicolumn{1}{c}{0.063}      & \multicolumn{1}{c}{0.061}         & \multicolumn{1}{c}{0.059}         & \multicolumn{1}{c}{0.086}         & 0.115                            & \multicolumn{1}{c}{0.051}      & \multicolumn{1}{c}{0.085}         & \multicolumn{1}{c}{0.159}         & \multicolumn{1}{c}{0.575}         & 0.792                            \\  
$m = 10$              & \multicolumn{1}{c}{0.065}      & \multicolumn{1}{c}{0.061}         & \multicolumn{1}{c}{0.068}         & \multicolumn{1}{c}{0.082}         & 0.155                            & \multicolumn{1}{c}{0.064}      & \multicolumn{1}{c}{0.115}         & \multicolumn{1}{c}{0.248}         & \multicolumn{1}{c}{0.735}         & 0.956                            \\ \bottomrule
\end{tabular}}
\end{table}

\begin{table}[!ht]
\centering
\setlength{\tabcolsep}{1.2mm}{ 
\caption{Case 2, linear transformation}
\label{tab:2.1}
\begin{tabular}{ccccccccccc}
\toprule
\multicolumn{1}{c}{} & \multicolumn{5}{c}{$\xi_n$ based}                                                                                                                                             & \multicolumn{5}{c}{distance correlation based} \\                                                                                                                                   \cmidrule(lr){2-6}\cmidrule(lr){7-11}
                       & \multicolumn{1}{c}{$\rho = 0$} & \multicolumn{1}{c}{$\rho = 0.05$} & \multicolumn{1}{c}{$\rho = 0.10$} & \multicolumn{1}{c}{$\rho = 0.15$} & \multicolumn{1}{c}{$\rho=0.20$} & \multicolumn{1}{c}{$\rho = 0$} & \multicolumn{1}{c}{$\rho = 0.05$} & \multicolumn{1}{c}{$\rho = 0.10$} & \multicolumn{1}{c}{$\rho = 0.15$} & \multicolumn{1}{c}{$\rho=0.20$} \\ \hline
$m = 1$               & \multicolumn{1}{c}{0.058}      & \multicolumn{1}{c}{0.065}         & \multicolumn{1}{c}{0.070}         & \multicolumn{1}{c}{0.119}         & 0.268                            & \multicolumn{1}{c}{0.060}      & \multicolumn{1}{c}{0.263}         & \multicolumn{1}{c}{0.767}         & \multicolumn{1}{c}{0.981}         & 1.000                            \\ 
$m = 2$               & \multicolumn{1}{c}{0.052}      & \multicolumn{1}{c}{0.054}         & \multicolumn{1}{c}{0.091}         & \multicolumn{1}{c}{0.275}         & 0.603                            & \multicolumn{1}{c}{0.048}      & \multicolumn{1}{c}{0.368}         & \multicolumn{1}{c}{0.952}         & \multicolumn{1}{c}{0.999}         & 1.000                            \\ 
$m = 3$               & \multicolumn{1}{c}{0.046}      & \multicolumn{1}{c}{0.066}         & \multicolumn{1}{c}{0.153}         & \multicolumn{1}{c}{0.440}         & 0.791                            & \multicolumn{1}{c}{0.047}      & \multicolumn{1}{c}{0.515}         & \multicolumn{1}{c}{0.986}         & \multicolumn{1}{c}{1.000}         & 1.000                            \\ 
$m = 5$               & \multicolumn{1}{c}{0.061}      & \multicolumn{1}{c}{0.065}         & \multicolumn{1}{c}{0.270}         & \multicolumn{1}{c}{0.664}         & 0.913                            & \multicolumn{1}{c}{0.049}      & \multicolumn{1}{c}{0.661}         & \multicolumn{1}{c}{0.997}         & \multicolumn{1}{c}{1.000}         & 1.000                            \\ 
$m = 10$              & \multicolumn{1}{c}{0.053}      & \multicolumn{1}{c}{0.069}         & \multicolumn{1}{c}{0.259}         & \multicolumn{1}{c}{0.528}         & 0.690                            & \multicolumn{1}{c}{0.052}      & \multicolumn{1}{c}{0.613}         & \multicolumn{1}{c}{0.991}         & \multicolumn{1}{c}{1.000}         & 1.000                            \\ \bottomrule
\end{tabular}}
\end{table}

\begin{table}[!ht]
\centering
\setlength{\tabcolsep}{1.2mm}{ 
\caption{Case 2, manifold transformation}
\label{tab:2.2}
\begin{tabular}{ccccccccccc}
\toprule
\multicolumn{1}{c}{} & \multicolumn{5}{c}{$\xi_n$ based}                                                                                                                                             & \multicolumn{5}{c}{distance correlation based} \\                                                                                                                                   \cmidrule(lr){2-6}\cmidrule(lr){7-11}
                       & \multicolumn{1}{c}{$\rho = 0$} & \multicolumn{1}{c}{$\rho = 0.05$} & \multicolumn{1}{c}{$\rho = 0.10$} & \multicolumn{1}{c}{$\rho = 0.15$} & \multicolumn{1}{c}{$\rho=0.20$} & \multicolumn{1}{c}{$\rho = 0$} & \multicolumn{1}{c}{$\rho = 0.05$} & \multicolumn{1}{c}{$\rho = 0.10$} & \multicolumn{1}{c}{$\rho = 0.15$} & \multicolumn{1}{c}{$\rho=0.20$} \\ \hline
$m = 1$               & \multicolumn{1}{c}{0.053}      & \multicolumn{1}{c}{0.051}         & \multicolumn{1}{c}{0.061}         & \multicolumn{1}{c}{0.118}         & 0.267                            & \multicolumn{1}{c}{0.055}      & \multicolumn{1}{c}{0.199}         & \multicolumn{1}{c}{0.671}         & \multicolumn{1}{c}{0.949}         & 0.998                            \\ 
$m = 2$               & \multicolumn{1}{c}{0.059}      & \multicolumn{1}{c}{0.054}         & \multicolumn{1}{c}{0.086}         & \multicolumn{1}{c}{0.206}         & 0.447                            & \multicolumn{1}{c}{0.046}      & \multicolumn{1}{c}{0.310}         & \multicolumn{1}{c}{0.841}         & \multicolumn{1}{c}{0.995}         & 1.000                            \\  
$m = 3$               & \multicolumn{1}{c}{0.068}      & \multicolumn{1}{c}{0.069}         & \multicolumn{1}{c}{0.115}         & \multicolumn{1}{c}{0.284}         & 0.539                            & \multicolumn{1}{c}{0.049}      & \multicolumn{1}{c}{0.345}         & \multicolumn{1}{c}{0.937}         & \multicolumn{1}{c}{0.999}         & 1.000                            \\ 
$m = 5$               & \multicolumn{1}{c}{0.077}      & \multicolumn{1}{c}{0.073}         & \multicolumn{1}{c}{0.138}         & \multicolumn{1}{c}{0.317}         & 0.529                            & \multicolumn{1}{c}{0.047}      & \multicolumn{1}{c}{0.477}         & \multicolumn{1}{c}{0.978}         & \multicolumn{1}{c}{1.000}         & 1.000                            \\ 
$m = 10$              & \multicolumn{1}{c}{0.080}      & \multicolumn{1}{c}{0.074}         & \multicolumn{1}{c}{0.127}         & \multicolumn{1}{c}{0.236}         & 0.334                            & \multicolumn{1}{c}{0.047}      & \multicolumn{1}{c}{0.651}         & \multicolumn{1}{c}{0.996}         & \multicolumn{1}{c}{1.000}         & 1.000                            \\ \bottomrule
\end{tabular}}
\end{table}

\begin{table}[!ht]
\centering
\setlength{\tabcolsep}{1.2mm}{ 
\caption{Case 3, linear transformation}
\label{tab:3.1}
\begin{tabular}{ccccccccccc}
\toprule
\multicolumn{1}{c}{} & \multicolumn{5}{c}{$\xi_n$ based}                                                                                                                                             & \multicolumn{5}{c}{distance correlation based} \\                                                                                                                                   \cmidrule(lr){2-6}\cmidrule(lr){7-11}
                       & \multicolumn{1}{c}{$\rho = 0$} & \multicolumn{1}{c}{$\rho = 0.05$} & \multicolumn{1}{c}{$\rho = 0.10$} & \multicolumn{1}{c}{$\rho = 0.15$} & \multicolumn{1}{c}{$\rho=0.20$} & \multicolumn{1}{c}{$\rho = 0$} & \multicolumn{1}{c}{$\rho = 0.05$} & \multicolumn{1}{c}{$\rho = 0.10$} & \multicolumn{1}{c}{$\rho = 0.15$} & \multicolumn{1}{c}{$\rho=0.20$} \\ \hline
$m = 1$               & \multicolumn{1}{c}{0.045}      & \multicolumn{1}{c}{0.053}         & \multicolumn{1}{c}{0.071}         & \multicolumn{1}{c}{0.134}         & 0.278                            & \multicolumn{1}{c}{0.050}      & \multicolumn{1}{c}{0.072}         & \multicolumn{1}{c}{0.240}         & \multicolumn{1}{c}{0.586}         & 0.914                            \\ 
$m = 2$               & \multicolumn{1}{c}{0.052}      & \multicolumn{1}{c}{0.044}         & \multicolumn{1}{c}{0.080}         & \multicolumn{1}{c}{0.222}         & 0.493                            & \multicolumn{1}{c}{0.049}      & \multicolumn{1}{c}{0.080}         & \multicolumn{1}{c}{0.211}         & \multicolumn{1}{c}{0.499}         & 0.815                            \\ 
$m = 3$               & \multicolumn{1}{c}{0.056}      & \multicolumn{1}{c}{0.048}         & \multicolumn{1}{c}{0.073}         & \multicolumn{1}{c}{0.189}         & 0.402                            & \multicolumn{1}{c}{0.051}      & \multicolumn{1}{c}{0.078}         & \multicolumn{1}{c}{0.196}         & \multicolumn{1}{c}{0.461}         & 0.648                            \\  
$m = 5$               & \multicolumn{1}{c}{0.054}      & \multicolumn{1}{c}{0.064}         & \multicolumn{1}{c}{0.056}         & \multicolumn{1}{c}{0.046}         & 0.050                            & \multicolumn{1}{c}{0.059}      & \multicolumn{1}{c}{0.099}         & \multicolumn{1}{c}{0.141}         & \multicolumn{1}{c}{0.375}         & 0.532                            \\ 
$m = 10$              & \multicolumn{1}{c}{0.059}      & \multicolumn{1}{c}{0.254}         & \multicolumn{1}{c}{0.476}         & \multicolumn{1}{c}{0.573}         & 0.620                            & \multicolumn{1}{c}{0.042}      & \multicolumn{1}{c}{0.056}         & \multicolumn{1}{c}{0.161}         & \multicolumn{1}{c}{0.327}         & 0.441                            \\ \bottomrule
\end{tabular}}
\end{table}

\begin{table}[!ht]
\centering
\setlength{\tabcolsep}{1.2mm}{ 
\caption{Case 3, manifold transformation}
\label{tab:3.2}
\begin{tabular}{ccccccccccc}
\toprule
\multicolumn{1}{c}{} & \multicolumn{5}{c}{$\xi_n$ based}                                                                                                                                             & \multicolumn{5}{c}{distance correlation based} \\                                                                                                                                   \cmidrule(lr){2-6}\cmidrule(lr){7-11}
                       & \multicolumn{1}{c}{$\rho = 0$} & \multicolumn{1}{c}{$\rho = 0.05$} & \multicolumn{1}{c}{$\rho = 0.10$} & \multicolumn{1}{c}{$\rho = 0.15$} & \multicolumn{1}{c}{$\rho=0.20$} & \multicolumn{1}{c}{$\rho = 0$} & \multicolumn{1}{c}{$\rho = 0.05$} & \multicolumn{1}{c}{$\rho = 0.10$} & \multicolumn{1}{c}{$\rho = 0.15$} & \multicolumn{1}{c}{$\rho=0.20$} \\ \hline
$m = 1$               & \multicolumn{1}{c}{0.047}      & \multicolumn{1}{c}{0.052}         & \multicolumn{1}{c}{0.068}         & \multicolumn{1}{c}{0.123}         & 0.264                            & \multicolumn{1}{c}{0.055}      & \multicolumn{1}{c}{0.068}         & \multicolumn{1}{c}{0.239}         & \multicolumn{1}{c}{0.514}         & 0.834                            \\  
$m = 2$               & \multicolumn{1}{c}{0.063}      & \multicolumn{1}{c}{0.058}         & \multicolumn{1}{c}{0.056}         & \multicolumn{1}{c}{0.082}         & 0.141                            & \multicolumn{1}{c}{0.047}      & \multicolumn{1}{c}{0.099}         & \multicolumn{1}{c}{0.247}         & \multicolumn{1}{c}{0.554}         & 0.886                            \\ 
$m = 3$               & \multicolumn{1}{c}{0.068}      & \multicolumn{1}{c}{0.088}         & \multicolumn{1}{c}{0.081}         & \multicolumn{1}{c}{0.069}         & 0.062                            & \multicolumn{1}{c}{0.049}      & \multicolumn{1}{c}{0.096}         & \multicolumn{1}{c}{0.285}         & \multicolumn{1}{c}{0.569}         & 0.799                            \\ 
$m = 5$               & \multicolumn{1}{c}{0.070}      & \multicolumn{1}{c}{0.130}         & \multicolumn{1}{c}{0.184}         & \multicolumn{1}{c}{0.204}         & 0.190                            & \multicolumn{1}{c}{0.040}      & \multicolumn{1}{c}{0.162}         & \multicolumn{1}{c}{0.221}         & \multicolumn{1}{c}{0.647}         & 0.819                            \\  
$m = 10$              & \multicolumn{1}{c}{0.078}      & \multicolumn{1}{c}{0.291}         & \multicolumn{1}{c}{0.519}         & \multicolumn{1}{c}{0.610}         & 0.673                            & \multicolumn{1}{c}{0.065}      & \multicolumn{1}{c}{0.059}         & \multicolumn{1}{c}{0.271}         & \multicolumn{1}{c}{0.604}         & 0.832                            \\ \bottomrule
\end{tabular}}
\end{table}

\begin{table}[!ht]
\centering
\setlength{\tabcolsep}{1.2mm}{ 
\caption{Case 4, linear transformation}
\label{tab:4.1}
\begin{tabular}{ccccccccccc}
\toprule
\multicolumn{1}{c}{} & \multicolumn{5}{c}{$\xi_n$ based}                                                                                                                                             & \multicolumn{5}{c}{distance correlation based} \\                                                                                                                                   \cmidrule(lr){2-6}\cmidrule(lr){7-11}
                       & \multicolumn{1}{c}{$\rho = 0$} & \multicolumn{1}{c}{$\rho = 0.05$} & \multicolumn{1}{c}{$\rho = 0.10$} & \multicolumn{1}{c}{$\rho = 0.15$} & \multicolumn{1}{c}{$\rho=0.20$} & \multicolumn{1}{c}{$\rho = 0$} & \multicolumn{1}{c}{$\rho = 0.05$} & \multicolumn{1}{c}{$\rho = 0.10$} & \multicolumn{1}{c}{$\rho = 0.15$} & \multicolumn{1}{c}{$\rho=0.20$} \\ \hline
$m = 1$                & \multicolumn{1}{c}{0.047}      & \multicolumn{1}{c}{0.080}         & \multicolumn{1}{c}{0.436}         & \multicolumn{1}{c}{0.892}         & 0.994                            & \multicolumn{1}{c}{0.050}      & \multicolumn{1}{c}{0.049}         & \multicolumn{1}{c}{0.062}         & \multicolumn{1}{c}{0.052}         & 0.070                            \\ 
$m = 2$                & \multicolumn{1}{c}{0.051}      & \multicolumn{1}{c}{0.044}         & \multicolumn{1}{c}{0.052}         & \multicolumn{1}{c}{0.057}         & 0.059                            & \multicolumn{1}{c}{0.051}      & \multicolumn{1}{c}{0.050}         & \multicolumn{1}{c}{0.048}         & \multicolumn{1}{c}{0.058}         & 0.046                            \\ 
$m = 3$                & \multicolumn{1}{c}{0.055}      & \multicolumn{1}{c}{0.056}         & \multicolumn{1}{c}{0.054}         & \multicolumn{1}{c}{0.051}         & 0.055                            & \multicolumn{1}{c}{0.051}      & \multicolumn{1}{c}{0.044}         & \multicolumn{1}{c}{0.055}         & \multicolumn{1}{c}{0.044}         & 0.043                            \\ 
$m = 5$                & \multicolumn{1}{c}{0.056}      & \multicolumn{1}{c}{0.053}         & \multicolumn{1}{c}{0.053}         & \multicolumn{1}{c}{0.052}         & 0.053                            & \multicolumn{1}{c}{0.058}      & \multicolumn{1}{c}{0.041}         & \multicolumn{1}{c}{0.037}         & \multicolumn{1}{c}{0.041}         & 0.046                            \\ 
$m = 10$               & \multicolumn{1}{c}{0.058}      & \multicolumn{1}{c}{0.054}         & \multicolumn{1}{c}{0.059}         & \multicolumn{1}{c}{0.057}         & 0.055                            & \multicolumn{1}{c}{0.044}      & \multicolumn{1}{c}{0.037}         & \multicolumn{1}{c}{0.081}         & \multicolumn{1}{c}{0.046}         & 0.056                            \\ \bottomrule
\end{tabular}}
\end{table}

\begin{table}[!ht]
\centering
\setlength{\tabcolsep}{1.2mm}{ 
\caption{Case 4, manifold transformation}
\label{tab:4.2}
\begin{tabular}{ccccccccccc}
\toprule
\multicolumn{1}{c}{} & \multicolumn{5}{c}{$\xi_n$ based}                                                                                                                                             & \multicolumn{5}{c}{distance correlation based} \\                                                                                                                                   \cmidrule(lr){2-6}\cmidrule(lr){7-11}
                       & \multicolumn{1}{c}{$\rho = 0$} & \multicolumn{1}{c}{$\rho = 0.05$} & \multicolumn{1}{c}{$\rho = 0.10$} & \multicolumn{1}{c}{$\rho = 0.15$} & \multicolumn{1}{c}{$\rho=0.20$} & \multicolumn{1}{c}{$\rho = 0$} & \multicolumn{1}{c}{$\rho = 0.05$} & \multicolumn{1}{c}{$\rho = 0.10$} & \multicolumn{1}{c}{$\rho = 0.15$} & \multicolumn{1}{c}{$\rho=0.20$} \\ \hline
$m = 1$               & \multicolumn{1}{c}{0.051}      & \multicolumn{1}{c}{0.073}         & \multicolumn{1}{c}{0.381}         & \multicolumn{1}{c}{0.864}         & 0.994                            & \multicolumn{1}{c}{0.044}      & \multicolumn{1}{c}{0.077}         & \multicolumn{1}{c}{0.170}         & \multicolumn{1}{c}{0.329}         & 0.580                            \\  
$m = 2$               & \multicolumn{1}{c}{0.055}      & \multicolumn{1}{c}{0.063}         & \multicolumn{1}{c}{0.062}         & \multicolumn{1}{c}{0.089}         & 0.123                            & \multicolumn{1}{c}{0.051}      & \multicolumn{1}{c}{0.073}         & \multicolumn{1}{c}{0.093}         & \multicolumn{1}{c}{0.114}         & 0.157                            \\ 
$m = 3$               & \multicolumn{1}{c}{0.081}      & \multicolumn{1}{c}{0.141}         & \multicolumn{1}{c}{0.166}         & \multicolumn{1}{c}{0.171}         & 0.168                            & \multicolumn{1}{c}{0.055}      & \multicolumn{1}{c}{0.056}         & \multicolumn{1}{c}{0.083}         & \multicolumn{1}{c}{0.073}         & 0.093                            \\ 
$m = 5$               & \multicolumn{1}{c}{0.079}      & \multicolumn{1}{c}{0.295}         & \multicolumn{1}{c}{0.404}         & \multicolumn{1}{c}{0.456}         & 0.474                            & \multicolumn{1}{c}{0.049}      & \multicolumn{1}{c}{0.093}         & \multicolumn{1}{c}{0.070}         & \multicolumn{1}{c}{0.095}         & 0.089                            \\  
$m = 10$              & \multicolumn{1}{c}{0.085}      & \multicolumn{1}{c}{0.462}         & \multicolumn{1}{c}{0.581}         & \multicolumn{1}{c}{0.621}         & 0.629                            & \multicolumn{1}{c}{0.042}      & \multicolumn{1}{c}{0.082}         & \multicolumn{1}{c}{0.055}         & \multicolumn{1}{c}{0.073}         & 0.067                            \\ \bottomrule
\end{tabular}}
\end{table}

\begin{table}[!ht]
\centering
\setlength{\tabcolsep}{1.2mm}{ 
\caption{Case 5, linear transformation}
\label{tab:6.1}
\begin{tabular}{ccccccccccc}
\toprule
\multicolumn{1}{c}{} & \multicolumn{5}{c}{$\xi_n$ based}                                                                                                                                             & \multicolumn{5}{c}{distance correlation based} \\                                                                                                                                   \cmidrule(lr){2-6}\cmidrule(lr){7-11}
                       & \multicolumn{1}{c}{$\rho = 0$} & \multicolumn{1}{c}{$\rho = 0.05$} & \multicolumn{1}{c}{$\rho = 0.10$} & \multicolumn{1}{c}{$\rho = 0.15$} & \multicolumn{1}{c}{$\rho=0.20$} & \multicolumn{1}{c}{$\rho = 0$} & \multicolumn{1}{c}{$\rho = 0.05$} & \multicolumn{1}{c}{$\rho = 0.10$} & \multicolumn{1}{c}{$\rho = 0.15$} & \multicolumn{1}{c}{$\rho=0.20$} \\ \hline
$m = 1$                & \multicolumn{1}{c}{0.048}      & \multicolumn{1}{c}{0.066}         & \multicolumn{1}{c}{0.272}         & \multicolumn{1}{c}{0.721}         & 0.971                            & \multicolumn{1}{c}{0.048}      & \multicolumn{1}{c}{0.068}         & \multicolumn{1}{c}{0.149}         & \multicolumn{1}{c}{0.323}         & 0.585                            \\ 
$m = 2$                & \multicolumn{1}{c}{0.047}      & \multicolumn{1}{c}{0.061}         & \multicolumn{1}{c}{0.334}         & \multicolumn{1}{c}{0.778}         & 0.955                            & \multicolumn{1}{c}{0.051}      & \multicolumn{1}{c}{0.046}         & \multicolumn{1}{c}{0.082}         & \multicolumn{1}{c}{0.106}         & 0.096                            \\ 
$m = 3$                & \multicolumn{1}{c}{0.061}      & \multicolumn{1}{c}{0.055}         & \multicolumn{1}{c}{0.097}         & \multicolumn{1}{c}{0.211}         & 0.311                            & \multicolumn{1}{c}{0.041}      & \multicolumn{1}{c}{0.050}         & \multicolumn{1}{c}{0.055}         & \multicolumn{1}{c}{0.058}         & 0.077                            \\ 
$m = 5$                & \multicolumn{1}{c}{0.052}      & \multicolumn{1}{c}{0.091}         & \multicolumn{1}{c}{0.107}         & \multicolumn{1}{c}{0.111}         & 0.106                            & \multicolumn{1}{c}{0.066}      & \multicolumn{1}{c}{0.069}         & \multicolumn{1}{c}{0.080}         & \multicolumn{1}{c}{0.067}         & 0.064                            \\ 
$m = 10$               & \multicolumn{1}{c}{0.055}      & \multicolumn{1}{c}{0.091}         & \multicolumn{1}{c}{0.121}         & \multicolumn{1}{c}{0.124}         & 0.138                            & \multicolumn{1}{c}{0.040}      & \multicolumn{1}{c}{0.037}         & \multicolumn{1}{c}{0.048}         & \multicolumn{1}{c}{0.041}         & 0.064                            \\ \bottomrule
\end{tabular}}
\end{table}

\begin{table}[!ht]
\centering
\setlength{\tabcolsep}{1.2mm}{ 
\caption{Case 5, manifold transformation}
\label{tab:6.2}
\begin{tabular}{ccccccccccc}
\toprule
\multicolumn{1}{c}{} & \multicolumn{5}{c}{$\xi_n$ based}                                                                                                                                             & \multicolumn{5}{c}{distance correlation based} \\                                                                                                                                   \cmidrule(lr){2-6}\cmidrule(lr){7-11}
                       & \multicolumn{1}{c}{$\rho = 0$} & \multicolumn{1}{c}{$\rho = 0.05$} & \multicolumn{1}{c}{$\rho = 0.10$} & \multicolumn{1}{c}{$\rho = 0.15$} & \multicolumn{1}{c}{$\rho=0.20$} & \multicolumn{1}{c}{$\rho = 0$} & \multicolumn{1}{c}{$\rho = 0.05$} & \multicolumn{1}{c}{$\rho = 0.10$} & \multicolumn{1}{c}{$\rho = 0.15$} & \multicolumn{1}{c}{$\rho=0.20$} \\ \hline
$m = 1$               & \multicolumn{1}{c}{0.047}      & \multicolumn{1}{c}{0.062}         & \multicolumn{1}{c}{0.267}         & \multicolumn{1}{c}{0.694}         & 0.956                            & \multicolumn{1}{c}{0.055}      & \multicolumn{1}{c}{0.061}         & \multicolumn{1}{c}{0.167}         & \multicolumn{1}{c}{0.348}         & 0.805                            \\  
$m = 2$               & \multicolumn{1}{c}{0.063}      & \multicolumn{1}{c}{0.056}         & \multicolumn{1}{c}{0.093}         & \multicolumn{1}{c}{0.158}         & 0.252                            & \multicolumn{1}{c}{0.047}      & \multicolumn{1}{c}{0.060}         & \multicolumn{1}{c}{0.114}         & \multicolumn{1}{c}{0.183}         & 0.303                            \\  
$m = 3$               & \multicolumn{1}{c}{0.068}      & \multicolumn{1}{c}{0.072}         & \multicolumn{1}{c}{0.070}         & \multicolumn{1}{c}{0.070}         & 0.072                            & \multicolumn{1}{c}{0.049}      & \multicolumn{1}{c}{0.054}         & \multicolumn{1}{c}{0.102}         & \multicolumn{1}{c}{0.136}         & 0.244                            \\ 
$m = 5$               & \multicolumn{1}{c}{0.070}      & \multicolumn{1}{c}{0.090}         & \multicolumn{1}{c}{0.110}         & \multicolumn{1}{c}{0.112}         & 0.110                            & \multicolumn{1}{c}{0.040}      & \multicolumn{1}{c}{0.101}         & \multicolumn{1}{c}{0.150}         & \multicolumn{1}{c}{0.130}         & 0.209                            \\ 
$m = 10$              & \multicolumn{1}{c}{0.078}      & \multicolumn{1}{c}{0.148}         & \multicolumn{1}{c}{0.186}         & \multicolumn{1}{c}{0.197}         & 0.196                            & \multicolumn{1}{c}{0.065}      & \multicolumn{1}{c}{0.038}         & \multicolumn{1}{c}{0.052}         & \multicolumn{1}{c}{0.126}         & 0.074                            \\ \bottomrule
\end{tabular}}
\end{table}


\section{Proofs}

Table \ref{tab:notation} lists all the symbols used in the following proofs.
\begin{table}[!ht]
\centering
\setlength{\tabcolsep}{3mm}{ 
\caption{Common symbols and notation}
\label{tab:notation}
\begin{tabular}{ll}
\hline
$\X$ & Random vector in $\mathbb{R}^d$ indicating the distribution of samples\\
$\X_i$,$\X_j$,$\X_k$... & i.i.d. copies of $\X$\\
$Y$ & Random vector in $\mathbb{R}$\\
$Y_i$,$Y_j$,$Y_k$... & i.i.d. copies of $Y$\\
$\mathcal{M}$ &$m$-dimensional smooth manifold in the Euclidean space $\mathbb{R}^d$\\
$m$ & Manifold dimension of $\mathcal{M}$\\
$d$ & Dimension of ambient space\\
$n$ & Number of data points\\
$\lambda$ & Lebesgue measure\\
$\|\cdot\|$ & Euclidean norm\\
$\mathcal{H}^m$ & $m$-dimensional Hausdorff measure\\
$\mu$ & Probability measure of $\X$\\
$g(\bx)$ & Density function of pushforward $\psi_{*}\mu$ at point $\bx$\\
$\mathcal{G}_n$ & Directed nearest neighbor graph (NNG) constructed on $n$ sample points $\X_1,...,\X_n$\\
$\mathcal{E}(\mathcal{G}_n)$ & Edge set of $\mathcal{G}_n$\\
$N_{total}$ & Total number of nearest pairs in $\mathcal{G}_n$\\
$N(\X_i)$ & Total number of nearest pairs in $\mathcal{G}_n$\\
$D(\X_i)$ & Out-degree of $\X_i$ in $\mathcal{G}_n$\\
$M_{total}$ & $\#\{(i,j,k) \enspace distinct: j\rightarrow i, k\rightarrow i \in \mathcal{E}(\mathcal{G}_n)\}$\\
$M(\X_i)$ & $\#\{(j,k) \enspace distinct: j\rightarrow i, k\rightarrow i \in \mathcal{E}(\mathcal{G}_n)\}$\\
$\pi$ & Tangent plane of a given point on $\mathcal{M}$\\
$U^d(\bx_1,\bx_2)$ & Union of two balls whose centers are $\bx_1$ and $\bx_2$ and radius are both $\Vert \bx_1-\bx_2 \Vert$ in $\mathbb{R}^d$\\
$B^k(\bx,r)$ & $k$-dimensional ball with center $\bx$ and radius $r$\\
$V_m$ & Volume of unit ball in $\R^m$\\
$U_m$ & Volume of $U^m(0,\rho)$, $\Vert \rho \Vert=1$\\
\hline
\end{tabular}}
\end{table}

\newpage{}

\subsection{Proof of Theorem \ref{main2}}
\begin{proof}[Proof of Theorem \ref{main2}]
Recall that we write $\mathcal{G}_n$ for the random directed nearest neighbor graph (NNG) corresponding to $n$ sample points $\X_1,...,\X_n$ and use $\mathcal{E}(\mathcal{G}_n)$ to denote the edge set of $\mathcal{G}_n$. According to Lemmas \ref{q} and \ref{o}:
\begin{align}
    \lim_{n\rightarrow\infty}&\mathrm{E}\bigg(\frac{1}{n}\#\Big\{(i,j)\enspace distinct: i\rightarrow j,j\rightarrow i\in \mathcal{E}(\mathcal{G}_n)\Big\}\bigg)=\kq_{m},\nonumber\\
    \lim_{n\rightarrow\infty}&\mathrm{E}\bigg(\frac{1}{n}\#\Big\{(i,j,k)\enspace distinct: i\rightarrow k,j\rightarrow k\in \mathcal{E}(\mathcal{G}_n)\Big\}\bigg)=\ko_{m}.\nonumber
\end{align}

Pursuing the idea in \cite{shi2021ac}, we resort to H\'{a}jek representation for calculating $\xi_n$'s asymptotic variance.
In \cite{lin2022limit}, the intermediate statistic $\widecheck{\xi}_n$ is defined as
\begin{align}
    \widecheck{\xi}_n := &\frac{6n}{n^2-1}\Big(\sum_{i=1}^n \min{\{F_{Y}(Y_i),F_{Y}(Y_{N(i)})\}}-\frac{1}{n-1}\sum_{i,j=1 \atop {i\neq j}}^n \min{\{F_{Y}(Y_i),F_{Y}(Y_j)\}}+\sum_{i=1}^n g(Y_i) \nonumber\\
    & + \frac{1}{n-1}\sum_{i,j=1 \atop {i\neq j}}^n E \Big[ \min{\{F_{Y}(Y_i),F_{Y}(Y_j)\}} \Big| \X_i,\X_j\Big]- E\bigg[\sum_{i=1}^n g(Y_i) \Big| \X_i\bigg] + \sum_{i=1}^n g_0(\X_i)\Big).\nonumber
\end{align}
where specific to the case that $Y$ is independent of $\X$, we have
\begin{align}
    F_{Y}(t) := \mathrm{P}(Y\leq t),\nonumber
\end{align}
\begin{equation}
    G_{\X}(t) := \mathrm{P}(Y \geq t |\X) = \mathrm{P}(Y \geq t)=1-F_{Y}(t), \nonumber
\end{equation}
\begin{equation}
    g(t) := \mathrm{Var}_{\X} \Big[ G_{\X}(t)\Big] =0, \nonumber
\end{equation}
\begin{equation}
    g_0(\bx) := \int\mathrm{E}\Big[ G_{\X}(t)\Big]^2\mathrm{d}F_{Y|\X=\bx}(t)=\int \big(1-F_Y(t)\big)^2 \mathrm{d} F_Y(t) = \frac{1}{3}. \nonumber
\end{equation}
Leveraging these expressions, $\widecheck{\xi}_n$ under independence then takes the form
\begin{equation}
    \widecheck{\xi}_n = \frac{6n}{n^2-1} \Big(\sum_{i=1}^n \min{\{F_{Y}(Y_i),F_{Y}(Y_{N(i)})\}}-\frac{1}{n-1}\sum_{i,j=1 \atop i\neq j}^n \min{\{F_{Y}(Y_i),F_{Y}(Y_j)\}}\Big) + C_0,\nonumber
\end{equation}
where $C_0$ is a fixed constant.

In \cite{lin2022limit}, it is proved that
\begin{equation}
    \lim_{n \to \infty} n \mathrm{Var}\big[\xi_n-\widecheck{\xi}_n\big] =0.\nonumber
\end{equation}
Thus we have
\begin{equation}
    \limsup_{n \to \infty}  \Big\vert \frac{\mathrm{Cov}(\xi_n,\xi_n - \widecheck{\xi}_n)}{\mathrm{Var}\big[\xi_n\big]}\Big\vert 
    \leq \limsup_{n \to \infty}\Big( \frac{\mathrm{Var}\big[\xi_n - \widecheck{\xi}_n\big]}{\mathrm{Var}\big[\xi_n\big]}\Big)^\frac{1}{2} 
    \leq\Big( \frac{\limsup_{n \to \infty}\mathrm{Var}\big[\xi_n - \widecheck{\xi}_n\big]}{\limsup_{n \to \infty}\mathrm{Var}\big[\xi_n\big]}\Big)^\frac{1}{2}
    = 0. \nonumber
\end{equation}
Then we have 
\begin{equation}\label{same}
    \lim_{n \to \infty} \frac{n \mathrm{Var}\big[\widecheck{\xi}_n\big]}{n \mathrm{Var}\big[\xi_n\big]}
    =\lim_{n \to \infty}\frac{\mathrm{Var}\big[\xi_n-(\xi_n-\widecheck{\xi}_n)\big]}{\mathrm{Var}\big[\xi_n\big]}
    =1.
\end{equation}
Since (\ref{same}) holds, it suffices to calculating the variance of the intermediate statistic $\widecheck{\xi}_n$, 
that is, we only have to calculate
\begin{equation}\label{var initial}
    n \mathrm{Var}\big[ \widecheck{\xi}_n\big] = \frac{36n^3}{(n^2-1)^2} 
    \mathrm{Var} \bigg[ \sum_{i=1}^n \min{\{F_{Y}(Y_i),F_{Y}(Y_{N(i)})\}}-\frac{1}{n-1}\sum_{i,j=1 \atop i\neq j}^n \min{\{F_{Y}(Y_i),F_{Y}(Y_j)\}}\bigg]. 
\end{equation}
For the sake of presentation clearness, introduce
\begin{equation}
    A_{ij}:= 6\min{\{F_{Y}(Y_i),F_{Y}(Y_j)\}}-2,\nonumber
\end{equation}
\begin{equation}
    V_i := n^{-\frac{1}{2}}\Big(\sum_{j:i\rightarrow j} A_{ij}-\frac{1}{n-1}\sum_{j:j\neq i} A_{ij}\Big),\quad {\rm and}~~~ S_n := \sum_{i=1}^n V_i.\nonumber \nonumber
\end{equation}
Then we can reformulate (\ref{var initial}) as :
\begin{equation}\label{var s}
    n \mathrm{Var}\big[ \widecheck{\xi}_n\big] =
    \Big(\frac{n^2}{n^2-1}\Big)^2 \mathrm{Var}\bigg[\sum_{i=1}^n V_i\bigg]
    =\Big(\frac{n^2}{n^2-1}\Big)^2 \mathrm{Var}\big[S_n\big].
\end{equation}
Since $Y_1,...,Y_n$ are independent and identically distributed (i.i.d.), for every $i\in\{1,...,n\}$, 
$F_{Y}(Y_i) \stackrel{i.i.d.}{\sim} \mathrm{Unif}[0,1]$. Thus, the expectation about $A_{ij}$ can be derived as
\begin{align*}
    & \mathrm{E} A_{ij}=0,~~ \mathrm{E} (A_{ij})^2=2:= \gamma_1, \quad \mathrm{E} (A_{ij}A_{ik})=\frac{4}{5}:= \gamma_2,~~{\rm and}~~ \mathrm{E} V_i=0.
\end{align*}
Get back to the calculation of $\mathrm{Var}\big[ \widecheck{\xi}_n\big]$. According to (\ref{var s}), it remains to calculate 
\begin{align}
    \mathrm{Var}\big[S_n\big] = \mathrm{E}\Big(\sum_{i=1}^n V_i\Big)^2= \sum_{i=1}^n \mathrm{E} V_i^2+\sum_{i\neq j}\mathrm{E} V_i V_j.\nonumber
\end{align}
For the first term, we have
\begin{align*}
\begin{split}
    \sum_{i=1}^n \mathrm{E} V_i^2 =&n^{-1} \sum_{i=1}^{n}\mathrm{E}\Big(\sum_{j:i\rightarrow j}A_{ij}-\frac{1}{n-1}\sum_{j:i\neq j}A_{ij}\Big)^2\\ 
    =& n^{-1}\bigg(\sum_{i=1}^n \mathrm{E}\Big(\sum_{i\rightarrow j} A_{ij}\Big)^2 +
    \sum_{i=1}^n \mathrm{E}\Big(\frac{1}{n-1}\sum_{j:i\neq j}A_{ij}\Big)^2\\
    &-\sum_{i=1}^n 2\mathrm{E}\Big(\frac{1}{n-1}(\sum_{i\rightarrow j} A_{ij})(\sum_{j:i\neq j}A_{ij})\Big)\bigg)\\
    = & \gamma_1+\Big(\frac{1}{n-1}\gamma_1+\frac{n-2}{n-1}\gamma_2\Big)-\Big(\frac{2}{n-1}\gamma_1+\frac{2(n-2)}{n-1}\gamma_2\Big)\\
    = & \Big(1-\frac{1}{n-1}\Big)\gamma_1-\Big(1-\frac{1}{n-1}\Big)\gamma_2\\
    \rightarrow & \gamma_1-\gamma_2.
\end{split}
\end{align*}

For the second term, we have
\begin{align*}
\begin{split}
    \sum_{i\neq j}\mathrm{E} V_i V_j
    =& n^{-1} \sum_{i\neq j}\mathrm{E}\Big(\sum_{k:i\rightarrow k}A_{ik}-\frac{1}{n-1}\sum_{l:i\neq l}A_{il}\Big)\Big(\sum_{m:j\rightarrow m}A_{jm}-\frac{1}{n-1}\sum_{p:j\neq p}A_{jp}\Big)\\
    =&n^{-1}\bigg(\sum_{i\neq j}\mathrm{E}\Big(\sum_{k:i\rightarrow k}A_{ik}\sum_{m:j\rightarrow m}A_{jm}\Big)
    +\sum_{i\neq j}\mathrm{E}\Big(\frac{1}{(n-1)^2}\sum_{l:i\neq l}A_{il}\sum_{p:j\neq p}A_{jp}\Big)\\
    &-\sum_{i\neq j}\mathrm{E}\Big(\frac{2}{n-1}\sum_{k:i\rightarrow k}A_{ik}\sum_{p:j\neq p}A_{jp}\Big)\bigg)\\
    =&\Big(\sum_{i\rightarrow j, j\rightarrow i}\gamma_1 + \sum_{i\rightarrow j, j\rightarrow i \atop {or i\rightarrow j, j\rightarrow k \atop or j\rightarrow i, i\rightarrow k}} \gamma_2\Big)+\Big(\frac{n}{n-1}\gamma_1+\frac{3(n-2)}{n-1}\gamma_2\Big)-\Big(\frac{n+1}{n-1}\gamma_1+\frac{6(n-2)}{n-1}\gamma_2\Big)\\
    = & -\frac{1}{n-1}\gamma_1-3\Big(1-\frac{1}{n-1}\Big)\gamma_2+\gamma_1\mathfrak{q}_{m}+\gamma_2(\mathfrak{o}_{m}+2-2\mathfrak{q}_{m})+o(1)\\
    \rightarrow &(\gamma_1-2\gamma_2)\kq_{m}+\gamma_2\ko_{m}-\gamma_2.
\end{split}
\end{align*}

In conclusion, we have
\begin{align}
    \lim_{n\rightarrow\infty}n \mathrm{Var}\left[ \widecheck{\xi}_n\right] = & \lim_{n\rightarrow\infty}\mathrm{Var}\big[S_n\big]\nonumber\\
    =& \lim_{n\rightarrow\infty}\sum_{i=1}^n \mathrm{E} V_i^2+\lim_{n\rightarrow\infty}\sum_{i\neq j}\mathrm{E} V_i V_j\notag\\
    =& \frac{2}{5}+\frac{2}{5}\kq_{m}+\frac{4}{5}\ko_{m}. \nonumber
\end{align}
This is the form of \citet[Theorem 3.1(ii)]{shi2021ac} when $\X_i$'s are sampled from an absolutely continuous distribution in $\mathbb{R}^m$.
\end{proof}



\subsection{Proof of Lemma \ref{al measure assumption}}
\begin{proof}[Proof of Lemma \ref{al measure assumption}]
Recall that $\mathcal{H}^m\mres\m$ represents the restricted $m$-dimensional Hausdorff measure on $\m$ and $\lambda$ represents the $m$-dimensional Lebesgue measure. 

To obtain the global property, i.e., $\mu$ is absolutely continuous with respect to $\mathcal{H}^m\mres\m$, it suffices to show the local property, i.e., for any point $\bx \in \m$ and the corresponding coordinate chart $(U,\psi)$, $\mu_U$ is absolutely continuous with respect to $\mathcal{H}^m\mres U$. In the following proof, we focus on a sufficiently small coordinate neighborhood $U$. Recall the definition of  absolutely continuity of $\mu$ with respect to some base measure, which states that any null set with respect to the base measure is a $\mu$-null set. With this definition, to prove Lemma \ref{al measure assumption}, it suffices to prove that any $\mathcal{H}^m$-null set $\psi^{-1}E$ is equivalent to a $\lambda$-null set $E$.

We refer to Chapter 3 of \cite{evans2018measure} for the following lemma and its proof. Assume a Lipschitz map $f:\mathbb{R}^m \rightarrow \mathbb{R}^n$, where $m\leq n$. Thus, $f$ is differentiable $\lambda$-a.e. by Rademacher's theorem (Theorem 3.2 in \cite{evans2018measure}). At any point $\by\in\mathbb{R}^m$ of differentiability we denote by $Jf(\by)$ the Jacobian of $f$. Also, we denote by $\mathcal{H}^0$ the 0-dimensinal Hausdorff measure, which is equivalent to counting measure.

\begin{lemma}\label{area formula}
    \emph{(Area formula).} We have $\by\mapsto Jf(\by)$  is Lebesgue measurable. In addition, for any Lebesgue measurable set $A\subset \mathbb{R}^m$, the map $z\mapsto f^{-1}(\{\bz\})$ is $\mathcal{H}^m$-measurable  and the following equality holds: 
    \begin{align*}
        \int_{\mathbb{R}^m}\mathcal{H}^{0}(A\cap f^{-1}(\bz))\ds \mathcal{H}^m(\bz)=\int_{A}Jf(\by)\ds \lambda(\by).
    \end{align*}
\end{lemma}

Back to the proof, since $\m$ is a $C^{\infty}$ manifold, the homeomorphisms $\psi$ and $\psi^{-1}$ are $C^{\infty}$ maps, thus locally Lipschitz continuous. Applying Lemma \ref{area formula}, for any $E\subset V$ and $\psi^{-1}(E) \subset U$ which is  $\mathcal{H}^m$-measurable, we have
\begin{align}\label{null set}
    \mathcal{H}^m(\psi^{-1}(E)) = \int_{E}Jf(\by)\ds \lambda(\by).
\end{align}
Here for a smooth manifold, $Jf(\by)$ can be written specifically as $Jf(\by) = \sqrt{\mathrm{det}\thinspace g}$, where $g$ is the metric tensor of the Riemannian submanifold, i.e., $g_{ij}=\partial_{i}\psi^{-1}\partial_{j}\psi^{-1}$.
Due to the existence of $\psi^{-1}$, the differentials of $\psi$ and $\psi^{-1}$ have maximum ranks everywhere locally, which implies that for every $\by \in E$, $Jf(\by) >0$. Using Equation \eqref{null set}, one directly obtains
\begin{align*}
    \mathcal{H}^m(\psi^{-1}(E)) = 0 ~~\text{ if and only if }~~ \lambda(E) = 0,
\end{align*}
which completes the proof.
\end{proof}

\subsection{Proof of Lemma \ref{q}}
\begin{proof}[Proof of Lemma \ref{q}]
Recall that $\mu$ is the (induced) probability measure of $\X$, that is, $\mu(A) = \mathbb{P}(\X \in A)$. Let $[\X_i]^{n}_{i=1}$ be a sample comprised of $n$ independent copies of $\X$. 
Let $N(\X_i)$ denote the number of NN pairs containing $\X_i$ and $N_{total}$ denote the total number of NN pairs (double edges) in NNG; in other words,  
\begin{align*}
    N_{total}:=&\#\Big\{(i,j) \enspace distinct: i\rightarrow j, j\rightarrow i \in \mathcal{E}(\mathcal{G}_n)\Big\},\\
    N(\X_i):=&\#\Big\{j: i\rightarrow j, j\rightarrow i \in \mathcal{E}(\mathcal{G}_n)\Big\}.
\end{align*}
Observe that
\begin{equation}
    N_{total} = \sum_{i=1}^n N(\X_i);\nonumber
\end{equation}
it thus suffices to calculate the value of $\E (N(\X_i))$. Using Lemma \ref{max degree}, there exists an upper bound of the number of points, whose nearest neighbor is $\X_i$. For any point $\bx_i$, we use constant $\mathfrak{C}_d$ to denote this upper bound:
\begin{equation}
    N(\bx_i) \leq \mathfrak{C}_d.\nonumber
\end{equation}
To obtain the asymptotic expecatation of $\E (N(X_i))$, it suffices to prove that for $\mu$-a.e.  $\bx_i \in \m$,
\begin{align}
    \lim_{n \to \infty}\mathrm{E}(N(\bx_i)) = \frac{V_m}{U_m},\nonumber
\end{align}
since the Lebesgue Dominated Convergence Theorem can be used as
\begin{align}
    \lim_{n \to \infty}\mathrm{E}N(\X_i) &= \lim_{n \to \infty} \int \mathrm{E}(N(\bx_i)) \ds\mu(\bx_i)\nonumber\\ 
    &= \int \lim_{n \to \infty}\mathrm{E}(N(\bx_i)) \ds\mu(\bx_i) = \frac{V_m}{U_m}.\nonumber
\end{align}

We use $U^d(\bx_1,\bx_2)$ to denote the the union of two balls in $\mathbb{R}^d$ whose centers are $\bx_1$ and $\bx_2$ and radius are both $\Vert \bx_1-\bx_2 \Vert$. Thus, $\bx_1$ and $\bx_2$ form a nearest neighbor pair if and only if there are no other sample points in $U^d(\bx_1,\bx_2)$.

We first consider the nearest neighbor of one fixed sample point $\X_i$ and write $N(\bx_i) := N(\X_i|\X_i = \bx_i)$ for convenience. Applying the absolute continuity assumption to the specific point $\bx_i\in\m$, we can write (recalling the notation of Assumption \ref{measure assump}) 
\[
g = \ds(\psi_{*}\mu)/\ds\lambda 
\]
for the density of the pushforward, that is $\psi_{*}\mu(A)=\int_{A}g\ds\lambda$, for any $\mu$-measurable set $A \in V$. 

We assume that $\bx_i\in \mathrm{supp}(\mu)$ and $g(\psi(\bx_i)) >0$ throughout the manuscript. Observe that for a fixed $\bx_i$,
\begin{align}
    \mathrm{E}(N(\bx_i)) =& \sum_{i \neq j} \mathrm{P}\Big(\bigcup_{k \neq i,j}\big\{\X_k \notin U^d(\bx_i,\X_j)\big\}\Big)\nonumber\\
    =&(n-1)\mathrm{P}\Big(\bigcup_{k \neq i,j}\big\{\X_k \notin U^d(\bx_i,\bX_j)\big\}\Big)\nonumber\\
    =&(n-1)\int_{\mathcal{M}}\big(1-\mu(U^d(\bx_i,\bx_j))\big)^{n-2}\mathrm{d} \mu(\bx_j),\label{main one}
\end{align}
which is the expression used for the following calculation.

\begin{lemma}\label{U lemma}
    For $\mu$-a.e. $\bx \in \m$ and the corresponding chart $\psi(=\psi_{\bx})$, it holds true that 
    \begin{align*}
        \lim_{\Vert \bx_j-\bx\Vert \rightarrow 0}\frac{\mu(U^d(\bx,\bx_j))}{g(\psi(\bx))\Vert \bx_j-\bx\Vert^m U_m} =1.
    \end{align*}
\end{lemma}
\begin{lemma}\label{B lemma}
     For $\mu$-a.e. $\bx \in \m$ and the corresponding chart $\psi(=\psi_{\bx})$, it holds true that
    \begin{align*}
        \lim_{\Vert \bx_j-\bx\Vert \rightarrow 0}\frac{\mu(B^d(\bx,r))}{g(\psi(\bx)) r^m V_m} =1.
    \end{align*}
\end{lemma}

Lemma \ref{U lemma} informs that when fixing $\bx_i \in \m$, for every $\epsilon >0$, it is almost sure that there exists $\delta_1 >0$ such that for every $x_j:\Vert \bx_j-\bx_i\Vert < \delta_1$, we have  
\begin{align}\label{inter balls}
    \frac{\mu(U^d(\bx_i,\bx_j))}{g(\psi(\bx_i))\Vert \bx_j-\bx_i\Vert^m U_m} \in [1-\epsilon,1+\epsilon].
\end{align}
Similarly, Lemma \ref{B lemma} informs that there exists $\delta_2 >0$ such that for every $r<\delta_2$,
\begin{align}\label{ball}
    \frac{\mu(B^d(\bx_i,r))}{g(\psi(\bx_i)) r^m V_m}\in [1-\epsilon,1+\epsilon].
\end{align}
In the following part of this section, we take $0<\delta <\min\{\delta_1,\delta_2\}$, which implies that both (\ref{inter balls}) and (\ref{ball}) hold for every $\bx_j$ such that $\Vert \bx_j-\bx_i\Vert < \delta$.

Back to the expression (\ref{main one}) for $\mathrm{E}(N(\bx_i))$, we first consider the lower bound,
\begin{align}
    \mathrm{E}(N(\bx_i)) &\geq (n-1)\int_{\Vert \bx_j-\bx_i\Vert\leq\delta}\Big(1-\mu\big(U^d(\bx_i,\bx_j)\cap \mathcal{M}\big)\Big)^{n-2}\ds \mu(\bx_j)\nonumber\\
    &\geq (n-1)\int_{\Vert \bx_j-\bx_i\Vert\leq\delta}\Big(1-(1+\epsilon) g(\psi(\bx_i))\Vert \bx_j-\bx_i\Vert^m U_m\Big)^{n-2}\ds\mu(\bx_j).\label{introphi}
\end{align}
Next, we introduce a nonincreasing function
\begin{equation}\label{defphi}
    \phi_1(r) := \big(1-(1+\epsilon) g(\psi(\bx_i))r^{m}U_m\big)^{n-2}I_{\{r\leq\delta\}},
\end{equation}
where $I_{\{r\leq\delta\}}$ denotes the indicator function. Notice that $\phi_1(0) = 1$.

With this function, we can reformulate Equation \eqref{introphi} as follows:
\begin{align}
    (\ref{introphi}) &= (n-1) \int_{\mathcal{M}}\phi_1\big(\Vert \bx_j-\bx_i \Vert\big)\ds\mu(\bx_j)\nonumber \\
    & = (n-1)\mathrm{E}\Big(\phi_1\big(\Vert \bX_j-\bx_i \Vert\big)\Big)\nonumber\\
    & = (n-1)\int_{0}^{1} \mathrm{P}\Big(\phi_1\big(\Vert \bX_j-\bx_i \Vert\big) > t\Big)\mathrm{d}t\nonumber\\
    & = (n-1)\int_{0}^{1}\Big(\int_{\Vert \bx_j-\bx_i \Vert \leq \phi_1^{-1}(t)} \ds\mu(\bx_j)\Big)\mathrm{d}t.\label{half}
\end{align}
For handling \eqref{half}, we select the range of $t$ to make $\bx_j$ close enough to $\bx_i$. Applying then the approximation derived above in (\ref{ball}) yields
\begin{align}
    (\ref{half})\geq& (n-1)\int_{\phi_1(\delta)}^{1}\Big(\int_{\Vert \bx_j-\bx_i \Vert \leq \phi_1^{-1}(t)} \ds\mu(\bx_j)\Big)\mathrm{d}t\nonumber\\
    =&(n-1)\int_{\phi_1(\delta)}^{1}\mu\Big(B^d\big(\bx_i,\phi_1^{-1}(t)\big)\Big)\ds t\nonumber\\
    \geq& (n-1)\int_{\phi_1(\delta)}^{1}(1-\epsilon) V_m \big(\phi_1^{-1}(t)\big)^m g(\psi_1(\bx_i)) \mathrm{d}t.\label{fs}
\end{align}
The expression of $\phi_1^{-1}(t)$ can be solved from Equation \eqref{defphi}. Plugging it into (\ref{fs}) gives
\begin{align}
    (\ref{fs})&= (n-1)\int_{\phi_1(\delta)}^{1} \frac{1-\epsilon}{1+\epsilon} \frac{V_m}{U_m}\big(1-t^{\frac{1}{n-2}}\big)\mathrm{d}t\nonumber \\
    &\geq \frac{1-\epsilon}{1+\epsilon} \frac{V_m}{U_m} \Big(1-(n-1)\phi_1(\delta)\Big).\label{final1}
\end{align}
Consider the limit of \eqref{final1}. First let $n$ go to infinity. Since
\begin{align}
    \lim_{n \rightarrow \infty}(n-1)\phi_1(\delta) \leq \lim_{n \rightarrow \infty}(n-1)\big(1-(1+\epsilon) g(\psi(\bx_i))\delta^{m}U_m\big)^{n-2}= 0,\nonumber
\end{align}
we have
\begin{equation}
    \lim_{n \to \infty}E(N(\bx_i)) \geq \frac{1-\epsilon}{1+\epsilon}\frac{V_m}{U_m}\nonumber
\end{equation}
holds for arbitrary $\epsilon > 0$. Thus, 
\begin{equation}
    \lim_{n \to \infty}E(N(\bx_i)) \geq \frac{V_m}{U_m}.\nonumber
\end{equation}

Now we turn to the inequality in the other direction to find the upper bound. Using the fact that 
\[
e^{-x} \geq 1-x, \text{ for any }x\in [0,1], 
\]
we have
\begin{align}
    \mathrm{E}(N(\bx_i)) &\leq (n-1)\int_{\mathcal{M}}\exp\Big(-(n-2)\mu\big(U^d(\bx_i,\bx_j)\cap \mathcal{M}\big)\Big)\ds\mu(\bx_j) \nonumber\\
    &= (n-1)\int_{\Vert \bx_j-\bx_i\Vert>\delta}\exp\Big(-(n-2)\mu\big(U^d(\bx_i,\bx_j)\cap \mathcal{M}\big)\Big)\ds\mu(\bx_j)\nonumber\\
    &+(n-1)\int_{\Vert \bx_j-\bx_i\Vert\leq\delta}\exp\Big(-(n-2)\mu\big(U^d(\bx_i,\bx_j)\cap \mathcal{M}\big)\Big)\ds\mu(\bx_j)\nonumber\\
    &=: I_{n,1} +I_{n,2}.\label{upper}
\end{align}
Here the domain of integration is partitioned into two sets. We use $I_{n,1}$ to denote the integral on $\m \cap \{\bx_j:\Vert \bx_j-\bx_i\Vert>\delta\}$ and $I_{n,2}$ to denote the integral on $\m \cap \{\bx_j:\Vert \bx_j-\bx_i\Vert\leq\delta\}$. We then process these two terms in (\ref{upper}) respectively.
For term $I_{n,1}$, we have
\begin{align}
    I_{n,1}=&(n-1)\int_{\Vert \bx_j-\bx_i\Vert>\delta} \exp\Big(-(n-2)\mu\big(U^d(\bx_i,\bx_j)\cap \mathcal{M}\big)\Big)\ds\mu(\bx_j)\nonumber\\
    \leq & (n-1)  \sup\limits_{\Vert \bx_j-\bx_i\Vert\geq\delta}  \Big\{\exp\Big(-(n-2)\mu\big(U^d(\bx_i,\bx_j)\cap \mathcal{M}\big)\Big)\Big\}\int_{\Vert \bx_j-\bx_i\Vert>\delta}\ds\mu(\bx_j) \nonumber\\
    \leq & (n-1)\exp\Big(-(n-2)\inf\limits_{\Vert \bx_j-\bx_i\Vert\geq\delta}\big\{\mu\big(U^d(\bx_i,\bx_j)\cap \mathcal{M}\big)\big\}\Big)\int_{\mathcal{M}}\ds\mu(\bx_j).\label{infi}
\end{align}
Since for $\nu > 0$, $U^d(\bx_i,\bx_j) \subset U^d(\bx_i,\bx_j+\nu(\bx_j-\bx_i))$ and $\int_{\mathcal{M}}\ds\mu=1$,
\begin{align}
    (\ref{infi})\leq&~ (n-1)\exp\Big(-(n-2)\inf\limits_{\Vert \bx_j-\bx_i\Vert=\delta}\mu\big(U^d(\bx_i,\bx_j)\cap \mathcal{M}\big)\Big)\nonumber\\
    :=&~ (n-1)\exp\big(-(n-2)h(\delta)\big).\nonumber
\end{align}
Here $h(\delta)$ is a positive function of $\delta$. Thus for any $\epsilon,\delta>0$, $\lim_{n\rightarrow\infty}I_{n,1}= 0$.

For the second term, using the approximation (\ref{inter balls}), we have
\begin{align}
    I_{n,2}=&(n-1)\int_{\Vert \bx_j-\bx_i\Vert\leq\delta}\exp\Big(-(n-2)\mu\big(U^d(\bx_i,\bx_j)\cap \mathcal{M}\big)\Big)\ds\mu(\bx_j)\nonumber\\
    \leq & (n-1)\int_{\Vert \bx_j-\bx_i\Vert\leq\delta}\exp\big(-(n-2)(1-\epsilon) g(\psi(\bx_i))\Vert \bx_i-\bx_j\Vert^m U_m\big)\ds\mu(\bx_j).\label{ks}
\end{align}
Again, introduce a nonincreasing function
\begin{equation}\label{defpsi}
    \phi_2(r) := \exp\big(-(n-2)(1-\epsilon)^2 g(\psi(\bx_i))r^{m}U_m\big) I_{\{r\leq \delta\}}.\nonumber
\end{equation}
and reformulate the expression as before:
\begin{align}
    (\ref{ks})& = (n-1)\mathrm{E}\big(\phi_2(\Vert \bX_j-\bx_i \Vert)\big)\nonumber\\
    & = (n-1)\int_{0}^{1} \mathrm{P}\big(\phi_2(\Vert \bX_j-\bx_i \Vert > t)\big)\mathrm{d}t\nonumber\\
    & = (n-1)\int_{0}^{1}\Big(\int_{\Vert \bx_j-\bx_i \Vert \leq \phi_2^{-1}(t)} \ds\mu(\bx_j)\Big)\mathrm{d}t\nonumber\\
    & \leq (n-1)\int_{0}^{\phi_2(\delta)}\ds t+(n-1)\int_{\phi_2(\delta)}^{1}\mu\Big(B^d\big(\bx_i,\phi_2^{-1}(t)\big)\Big) \ds t\nonumber \\
    & \leq (n-1)\phi_2(\delta)+(n-1)\int_{\phi_2(\delta)}^{1}(1+\epsilon) V_m \big(\phi_2^{-1}(t)\big)^m \ds t.\label{hs}
\end{align}
The expression of $\phi_2^{-1}(t)$ can be solved from \eqref{defpsi}. Then the integral above can be specifically calculated as
\begin{align}
    (\ref{hs})= & (n-1)\phi_2(\delta)+\frac{n-1}{n-2}\frac{1+\epsilon}{1-\epsilon}\frac{V_m}{U_m}\Big(1-\phi_2(\delta)+\phi_2(\delta)\log\phi_2(\delta)\Big).\nonumber
\end{align}
For fixed $\epsilon>0$ and $\delta>0$, one can use the definition of $\phi_2(r)$ and get 
\begin{align}
    \lim_{n \to \infty}n\phi_2(\delta) = 0.\nonumber
\end{align} 
Then we have
\begin{align}
    \lim_{n \to \infty}\mathrm{E}(N(\bx_i)) \leq \frac{1+\epsilon}{1-\epsilon}\frac{V_m}{U_m}\nonumber
\end{align}
holds for arbitrary $\epsilon>0$. Thus,
\begin{align}
    \lim_{n \to \infty}\mathrm{E}(N(\bx_i)) \leq \frac{V_m}{U_m}.\nonumber
\end{align}

Combining the upper and lower bounds yields
\begin{align}
    \lim_{n \to \infty}\mathrm{E}(N(\bx_i)) = \frac{V_m}{U_m}.\nonumber
\end{align}

Finally, we obtain the conclusion
\begin{align}
    \lim_{n \to \infty}\frac{\mathrm{E}N_{total}}{n} &= \lim_{n \to \infty}\sum_{i=1}^n \frac{\mathrm{E}(N(\X_i))}{n}=\frac{V_m}{U_m}=\kq_{m}.\nonumber
\end{align}

This completes the proof.
\end{proof}

\subsection{Proof of Lemma \ref{o}}

\begin{proof}[Proof of Lemma \ref{o}]
Analogous to the proof of Lemma \ref{q}, we only have to consider one specific sample point $\X_i$. Let $Q_j$ denote the event that $\X_i$ is the nearest neighbor of $\X_j$, that is, $Q_j := \big\{ j\rightarrow i \in \mathcal{E}(\mathcal{G}_n) \big\}$ for all $j \neq i$ and put 
\begin{align}
    M_{total}&=\#\Big\{(i,j,k) \enspace distinct: j\rightarrow i, k\rightarrow i \in \mathcal{E}(\mathcal{G}_n)\Big\},\nonumber\\
    D(\X_i) &= \sum_{j \neq i} I_{Q_j} = \#\Big\{ j:j\rightarrow i \in \mathcal{E}(\mathcal{G}_n)\Big\},\nonumber \\
    M(\X_i) &= \#\Big\{(j,k) \enspace distinct: j\rightarrow i, k\rightarrow i \in \mathcal{E}(\mathcal{G}_n)\Big\}.\nonumber
\end{align}
Here $M(\X_i)$ is the coefficient we are interested in. Observe that
\begin{align}
    M_{total}=&\sum_{i=1}^n M(\X_i),\nonumber\\
    M(\X_i) =& D(\X_i)(D(\X_i)-1),\nonumber
\end{align}
so that we have
\begin{align}
    \mathrm{E}M(\X_i) &= \mathrm{E}D(\X_i)(D(\X_i)-1)\nonumber\\
    & =\mathrm{E}\sum_{(j,k)\enspace distinct \atop j,k\neq i} \mathrm{P}(Q_j \cap Q_k)\nonumber\\
    & =(n-1)(n-2)\mathrm{P}(Q_j \cap Q_k).\nonumber
\end{align}
We first consider a fixed sample point $\X_i = \bx_i$. Using Lemma \ref{max degree}, we have
\begin{align}
    \mathfrak{C}_{d}^2 &\geq \mathrm{E}D\big(\X_i|\X_i=\bx_i\big)\big(D\big(\X_i|\X_i=\bx_i\big)-1\big)\nonumber\\
    &=(n-1)(n-2)\mathrm{P}\big(Q_j \cap Q_k|\X_i=\bx_i\big) .\nonumber
\end{align}
Then applying Lebesgue Dominated Convergence Theorem, we have
\begin{align}
    \lim_{n \to \infty} (n-1)(n-2)\mathrm{P}(Q_j \cap Q_k)&=\lim_{n \to \infty}\int_{\mathcal{M}}(n-1)(n-2)\mathrm{P}\big(Q_j \cap Q_k|\X_i=\bx_i\big) \ds\mu(\bx_i) \nonumber\\
    &=\int_{\mathcal{M}}\lim_{n \to \infty}(n-1)(n-2)\mathrm{P}\big(Q_j \cap Q_k|\X_i=\bx_i\big) \ds\mu(\bx_i).\nonumber
\end{align}
It thus remains to prove that
\begin{align}
    \lim_{n \to \infty} n^2\mathrm{P}\big(Q_j \cap Q_k|\X_i=\bx_i\big) =\lim_{n \to \infty} (n-1)(n-2)\mathrm{P}\big(Q_j \cap Q_k|\X_i=\bx_i\big) = \ko_m,\nonumber
\end{align}
where $\bx_i$ is any point with positive density, which is held fixed in what follows.

We first introduce the notation to simplify the integral. For a fixed point $\bx \in \mathcal{M}$, let 
\begin{align}
    \Gamma_x &:= \Big\{(\bx_j,\bx_k) \in (\mathbb{R}^d)^2:\mathrm{max}\{\Vert \bx_j-\bx\Vert,\Vert \bx_k-\bx\Vert \}\leq \Vert \bx_j-\bx_k\Vert\Big\}, \nonumber\\
    & S_j := B\Big(\bx_j,\Vert \bx_j-\bx \Vert\Big)\bigcap\mathcal{M}, \quad S_k := B\Big(\bx_k,\Vert \bx_k-\bx \Vert\Big)\bigcap\mathcal{M}.\nonumber
\end{align}
Applying the above notation, we have that $\X=\bx$ is the nearest neighbor of $\X_j=\bx_j$ if and only if there are no other sample points in $S_j$. In the following part, we record $\bx_i$ as $\bx$ for notation simplicity. It then holds true that
\begin{align}
    n^2\mathrm{P}\big(Q_j \cap Q_k|\X_i=\bx\big) = \iint_{\Gamma_x} n^2 \big(1-\mu\left( S_j\cup S_k\right)\big)^{n-2}\ds\mu(\bx_j)\ds\mu(\bx_k),\nonumber
\end{align}
where we use $\ds\mu(\bx_j)$ and $\ds\mu(\bx_k)$ to denote the measure and the corresponding random variables which the integral corresponds to.
The next step is to split the region of integral that $\Gamma_{\bx} = \Gamma_{\bx,\delta}^1\cup\Gamma_{\bx,\delta}^2$. Here
\begin{align}
    \Gamma_{\bx,\delta}^1 &:= \Gamma_{\bx} \bigcap \Big\{ (\bx_1,\bx_2)\in (\mathbb{R}^d)^2:\Vert \bx_1-\bx\Vert,\Vert \bx_2-\bx\Vert\leq \delta\Big\},\nonumber\\
    \Gamma_{\bx,\delta}^2 &:=\Gamma_{\bx} \setminus\Gamma_{\bx,\delta}^1,\nonumber
\end{align}
where $\delta$ is a sufficiently small but positive real number. Specifically, $\delta=\min\{\delta_1,\delta_2,\delta_3,\delta_4,\delta_5\}$, in which $\delta_i$'s are to be defined in the proof below.
Taking this idea, we partition the integral into two parts:
\begin{align}
    n^2\mathrm{P}(Q_j \cap Q_k|\X_i=\bx)
    =&\iint_{\Gamma_{\bx,\delta}^1} n^2 \big(1-\mu\left( S_j\cup S_k\right)\big)^{n-2}\ds\mu(\bx_j)\ds\mu(\bx_k)\nonumber\\
    &+\iint_{\Gamma_{\bx,\delta}^2} n^2 \big(1-\mu\left( S_j\cup S_k\right)\big)^{n-2}\ds\mu(\bx_j)\ds\mu(\bx_k)\nonumber\\
    =:& J_{n,1}+J_{n,2}.\nonumber
\end{align}
We first prove that
\begin{equation}
    \lim_{n\rightarrow\infty}J_{n,2} = 0.\nonumber
\end{equation}

\begin{lemma}\label{B2 lemma}
    For $\mu$-a.e. $\bx \in \m$ and the corresponding chart $\psi$, it holds true that
    \begin{align*}
        \lim_{\Vert \bx_j-\bx\Vert \rightarrow 0}\frac{\mu\big(B^d(\bx_j,\Vert \bx-\bx_j\Vert))\big)}{g(\psi(\bx)) \lambda\big(B^m(\bx_j,\Vert \bx-\bx_j\Vert)\big)} =1.
    \end{align*}
\end{lemma}

Lemma \ref{B2 lemma} shows that for any $\epsilon >0$, there exists $\delta_1 >0$ such that for $\mu$-a.e. $\bx \in \mathcal{M}$ and for any $\bx_j$ such that $\Vert \bx-\bx_j\Vert\leq\delta_1$,
\begin{align}
    &\frac{\mu\big(B^d(\bx_j,\Vert \bx-\bx_j\Vert))\big)}{g(\psi(\bx)) \lambda\big(B^m(\bx_j,\Vert \bx-\bx_j\Vert)\big)} \in \big[1-\epsilon,1+\epsilon\big].\nonumber
\end{align}
By definition of $\Gamma_{\bx,\delta}^2$, for any $(\bx_j,\bx_k) \in \Gamma_{\bx,\delta}^2$, there exists $\bx_l= \bx_j$ or $\bx_k$ satisfying $\Vert \bx_l -\bx \Vert > \delta$. Without loss of generality, we assume that $\bx_l=\bx_j$. For the point
\begin{equation}
    \by := \bx- \frac{\delta}{2}\frac{\bx-\bx_j}{\Vert \bx-\bx_j \Vert},\nonumber
\end{equation}
we have $\Vert \bx-\by \Vert = \delta/2$. For every $\bx^{*}\in B(\by,\Vert \bx-\by\Vert)=B(\by,\frac{\delta}{2})$,
\begin{align}
    \Vert \bx^{*}-\bx_j\Vert &\leq \Vert \bx^{*}-\by\Vert+\Vert \by-\bx_j\Vert < \frac{\delta}{2}+\Big(\Vert \bx-\bx_j\Vert-\frac{\delta}{2}\Big) = \Vert \bx-\bx_j\Vert,\nonumber
\end{align}
yielding
\begin{equation}
    B\Big(\by,\frac{\delta}{2}\Big) \subset B(\bx,\delta) \cap B(\bx_j,\Vert \bx_j-\bx\Vert).\nonumber
\end{equation}
This in turn implies
\begin{align}
    \mu\left( S_j\cup S_k\right) &\geq \mu\big(B(\bx_j,\Vert \bx_j-\bx\Vert)\cap\mathcal{M}\big)\nonumber\\
    &\geq \mu\big(B(\by,\delta/2)\cap\mathcal{M}\big)\nonumber\\
    &\geq (1-\epsilon)g(\psi(\bx))\lambda\big(B^m(0,\delta/2)\big)\nonumber\\
    &=(1-\epsilon)g(\psi(\bx))V_m\Big(\frac{\delta}{2}\Big)^m.\nonumber
\end{align}
Back to the integral $J_{n,2}$, we then have the upper bound
\begin{align}
    n^2 \big(1-\mu\left( S_j\cup S_k\right)\big)^{n-2} \leq n^2\Big(1-(1-\epsilon)g(\psi(\bx))V_m\Big(\frac{\delta}{2}\Big)^m\Big)^{n-2},\nonumber
\end{align}
which is a constant with respect to $\bx_j$ and $\bx_k$ and thus 
\begin{align}
    0 \leq \lim_{n\rightarrow\infty}J_{n,2} \leq \lim_{n\rightarrow\infty}n^2\Big(1-(1-\epsilon)g(\psi(\bx))V_m\Big(\frac{\delta}{2}\Big)^m\Big)^{n-2} = 0.\nonumber
\end{align}
This finishes the proof of the first part.

We then prove that
\begin{equation}
    \lim_{n\rightarrow\infty}J_{n,1} = \ko_{m}, \nonumber
\end{equation}
that is, 
\begin{align}
    \lim_{n\rightarrow\infty}\iint_{\Gamma_{\bx,\delta}^1} n^2 \big(1-\mu\left( S_j\cup S_k\right)\big)^{n-2}\ds\mu(\bx_j)\ds\mu(\bx_k)=\ko_{m}.\nonumber
\end{align}
Here $\delta$ is a constant we will select later. Similar to the process in the proof of Lemma \ref{q}, we will then derive the upper and lower bounds of $J_{n,1}$.

For the upper bound, first leveraging the fact that $e^{-x} \geq 1-x$ for all $x\in (0,1)$, we have
\begin{align}\label{eq:FH1}
    J_{n,1}\leq \iint_{\Gamma_{\bx,\delta}^1} n^2 \exp\Big(-(n-2)\mu\left( S_j\cup S_k\right)\Big)\ds\mu(\bx_j)\ds\mu(\bx_k).
\end{align}
We set $\bz_j=\psi(\bx_j)$, $\bz_k=\psi(\bx_k)$ for simplicity and thus, $\bx_j=\psi^{-1}(\bz_j)$, $\bx_k=\psi^{-1}(\bz_k)$.
\begin{lemma}\label{B2cup lemma}
    For $\mu$-a.e. $\bx \in \m$ and the corresponding chart $\psi$, we have
    \begin{align*}
        \lim_{\Vert \bx_j-\bx\Vert \rightarrow 0 \atop {\Vert \bx_k-\bx\Vert \rightarrow 0}}\frac{\mu\Big((B^d(\bx_j,\Vert \bx_j-\bx \Vert)\cup B^d(\bx_k,\Vert \bx_k-\bx \Vert))\Big)}{g(\psi(\bx))\lambda\big(B^m(\bz_j,\Vert \bz_j-\psi(\bx) \Vert)\cup B^m(\bz_k,\Vert \bz_k-\psi(\bx) \Vert)\big)} =1.
    \end{align*}
\end{lemma}

Lemma \ref{B2cup lemma} implies that for $\mu$-a.e $\bx \in \Gamma_{\bx}$ and any $\epsilon >0$, there exists $\delta_3>0$ such that for any $(\bx_j,\bx_k)$ satisfying $\Vert \bx_j-\bx \Vert \leq \delta_3,\Vert \bx_k-\bx \Vert \leq \delta_3$, we have
\begin{align*}
    \frac{\mu\Big((B^d(\bx_j,\Vert \bx_j-\bx \Vert)\cup B^d(\bx_k,\Vert \bx_k-\bx \Vert))\Big)}{g(\psi(\bx))\lambda\big(B^m(\bz_j,\Vert \bz_j-\psi(\bx) \Vert)\cup B^m(\bz_k,\Vert \bz_k-\psi(\bx) \Vert)\big)}\in [1-\epsilon,1+\epsilon].
\end{align*}

Plugging them into the righthand side of \eqref{eq:FH1}, we obtain
\begin{align}
    &\iint_{\Gamma_{\bx,\delta}^1} n^2 \exp\Big(-(n-2)\mu\left( S_j\cup S_k\right)\Big)\ds\mu(\bx_j)\ds\mu(\bx_k)\nonumber\\
    \leq& \iint_{\Gamma_{\bx,\delta}^1}\!\!n^2\exp\bigg(-(n-2)(1-\epsilon)g(\psi(\bx))\lambda\Big(B^m(\bz_j,\Vert \bz_j-\psi(\bx) \Vert)\cup B^m(\bz_k,\Vert \bz_k-\psi(\bx) \Vert)\Big)\bigg)\nonumber\\
    &\qquad \qquad \qquad \qquad \qquad \qquad \qquad \qquad \qquad \qquad \qquad \qquad \qquad \qquad \qquad \qquad \qquad\ds\mu(\bx_j)\ds\mu(\bx_k).\nonumber\\
    =&\iint_{\psi(\Gamma_{\bx,\delta}^1)}\!\!n^2\exp\bigg(-(n-2)(1-\epsilon)g(\psi(\bx))\lambda\Big(B^m(\bz_j,\Vert \bz_j-\psi(\bx) \Vert)\cup B^m(\bz_k,\Vert \bz_k-\psi(\bx) \Vert)\Big)\bigg)\nonumber\\
    &\qquad \qquad \qquad \qquad \qquad \qquad \qquad \qquad \qquad \qquad \qquad \qquad \qquad \qquad \qquad
    g(\bz_j)g(\bz_k) \ds\lambda(\bz_j)\ds\lambda(\bz_k).\label{trans}
\end{align}

\begin{lemma}\label{newLDT}
Assume all the conditions of Lebesgue Differential Theorem hold, and $f$, $g$ are integrable functions with respect to Lebesgue measure $\lambda$. We then have
\begin{align*}
    \lim_{U\rightarrow x, U\in\mathcal{V}}\frac{1}{\vert U\vert}\int_{U}f g\ds\lambda=f(\bx)\lim_{U\rightarrow x, U\in\mathcal{V}}\frac{1}{\vert U\vert}\int_{U} g\ds\lambda.
\end{align*}
\end{lemma} 

Using Lemma \ref{newLDT}, for fixed $\epsilon>0$, there exists $\delta_4 >0$ such that for every $\delta<\delta_4$, we have
\begin{align}
    (\ref{trans})\leq (1+\epsilon)n^2 g(\psi(\bx))^2 &\iint_{\psi(\Gamma_{\bx,\delta}^1)}\!\!\exp\bigg(-(n-2)(1-\epsilon)g(\psi(\bx))\nonumber\\
    & \lambda\Big(B^m(\bz_j,\Vert \bz_j-\psi(\bx) \Vert)\cup B^m(\bz_k,\Vert \bz_k-\psi(\bx) \Vert)\Big)\bigg) \ds\lambda(\bz_j)\ds\lambda(\bz_k).\label{ls}
\end{align}

\begin{lemma}\label{area lemma}
    For $\mu$-a.e. $\bx \in \m$ and the corresponding chart $\psi$, it holds true that
    \begin{align*}
        \lim_{\delta \rightarrow 0}\frac{\lambda(\psi(\Gamma_{\bx,\delta}^1))}{\lambda(\Gamma_{m,\delta})} =1.
    \end{align*}
\end{lemma}

Considering Lemma \ref{area lemma} and the uniform property of Lebesgue measure, we can just do a translation and make the origin point contained in the domain of integration so that
\begin{align}
    (\ref{ls})\leq (1+\epsilon)^2 n^2 g(\psi(\bx))^2 \iint_{\Gamma_{m,\delta}}\!\!\exp\Big(-(n-2)&(1-\epsilon) g(\psi(\bx))\nonumber\\
    &\lambda\big(B^m(\bz_j,\Vert \bz_j \Vert)\cup B^m(\bz_k,\Vert \bz_k \Vert)\big) \ds\lambda(\bz_j)\ds\lambda(\bz_k),\nonumber
\end{align}
where 
\[
\Gamma_{m,\delta}:=\Gamma_m\bigcap\Big\{ (\bz_j,\bz_k)\in (\mathbb{R}^m)^2:\Vert \bz_j\Vert,\Vert \bz_k\Vert\leq \delta\Big\}. 
\]
Denote $B^m(\bz_j,\Vert \bz_j \Vert)$ and $B^m(\bz_k,\Vert \bz_k \Vert)$ to be $m$-dimensional balls in $\mathbb{R}^m$. We then introduce a translation
\begin{align}
    \bz_j^{*} := l_m \bz_j,\quad \bz_k^{*}:=l_m \bz_k.\nonumber
\end{align}
We thus have
\begin{align}
    &\lambda\Big(B^m(\bz_j^{*},\Vert \bz_j^{*} \Vert)\cup(B^m(\bz_k^{*},\Vert \bz_k^{*} \Vert)\Big) = (l_m)^m \lambda\Big(B^m(\bz_j,\Vert \bz_j \Vert)\cup(B^m(\bz_k,\Vert \bz_k \Vert)\Big).\nonumber
\end{align}
Set $l_m = \Big((n-2)(1-\epsilon)g(\psi(\bx))\Big)^{1/m}$ and the corresponding Jacobi matrix satisfies $\vert J\vert = (n-2)(1-\epsilon) g(\psi(\bx))$. As $n \rightarrow \infty$, $l_m \rightarrow \infty$ and $\frac{n}{n-2} \rightarrow 1$, and thus
\begin{align}
    \lim_{n \rightarrow \infty}J_{n,1}&\leq\lim_{n \rightarrow \infty} \frac{(1+\epsilon)n^2}{(1-\epsilon)^2(n-2)^2}  \iint_{\Gamma_{m,l_m\delta}}\!\!\exp\bigg(-\lambda\Big(B^m(\bz_j^{*},\Vert \bz_j^{*} \Vert)\cup(B^m(\bz_k^{*},\Vert \bz_k^{*} \Vert)\Big)\bigg)\mathrm{d}\lambda(\bz_j^{*})\mathrm{d}\lambda(\bz_k^{*})\nonumber\\
    &= \frac{1+\epsilon}{(1-\epsilon)^2}\iint_{\Gamma_m}\exp\Big(-\lambda\big(B^m(\bz_j^{*},\Vert \bz_j^{*} \Vert)\cup(B^m(\bz_k^{*},\Vert \bz_k^{*} \Vert)\big)\Big)\ds\lambda(\bz_j^{*})\ds\lambda(\bz_k^{*})\nonumber
\end{align}
holds for arbitrary $\epsilon > 0$. We thus obtain the upper bound.

For the lower bound, the proof is similar. Use the fact that, for any $\epsilon >0$, there exists $\sigma >0$ such that for every $x \in [0,\sigma]$, $e^{-x} \leq (1+\epsilon)(1-x)$. We have there exists $\delta_5$ such that for any $\delta < \delta_5$,
\begin{align}
    J_{n,1}\geq \iint_{\Gamma_{\bx,\delta}^1} n^2 \exp\Big(-(n-2)(1+\epsilon)\mu\left( S_j\cup S_k\right)\Big)\ds\mu(\bx_j)\ds\mu(\bx_k).\nonumber
\end{align}
We then follow the same process as above: for a sufficiently small $\delta>0$,
\begin{align}
    &\iint_{\Gamma_{\bx,\delta}^1} n^2 \exp(-(n-2)(1+\epsilon)\mu\left( S_j\cup S_k\right))\ds\mu(\bx_j)\ds\mu(\bx_k)\nonumber\\
    \geq & \iint_{\Gamma_{\bx,\delta}^1} n^2 \exp\bigg(-(n-2)(1+\epsilon)^2 g(\psi(\bx))\nonumber\\
    & \qquad \qquad \qquad \qquad  \qquad \qquad\lambda\Big(B^m(\bz_j,\Vert \bz_j-\psi(\bx) \Vert)\cup B^m(z_k,\Vert z_k-\psi(\bx) \Vert)\Big)\bigg)\ds\mu(\bx_j)\ds\mu(\bx_k)\nonumber\\
    \geq & (1-\epsilon)n^2 g(\psi(\bx))^2 \iint_{\psi(\Gamma_{\bx,\delta}^1)}\!\!\exp\bigg(-(n-2)(1+\epsilon)^2 g(\psi(\bx))\nonumber\\
    &\qquad \qquad \qquad \qquad  \lambda\Big(B^m(\bz_j,\Vert \bz_j-\psi(\bx) \Vert)\cup B^m(\bz_k,\Vert \bz_k-\psi(\bx) \Vert)\Big)\bigg) g(\bz_j)g(\bz_k) \ds\lambda(\bz_j)\ds\lambda(\bz_k)\nonumber\\
    \geq & (1-\epsilon)^2 n^2 g(\psi(\bx))^2 \iint_{\Gamma_{m,\delta}}\!\exp\bigg(-(n-2)(1+\epsilon)^2 g(\psi(\bx))\nonumber\\
    &\qquad \qquad \qquad \qquad \qquad \qquad \qquad \qquad \qquad
    \lambda\Big(B^m(\bz_j,\Vert \bz_j \Vert)\cup(B^m(\bz_k,\Vert \bz_k \Vert)\Big)\bigg)\ds\lambda(\bz_j)\ds\lambda(\bz_k).\nonumber
\end{align}
Again use a translation
\begin{align}
    \bz_j^{**} := r_m \bz_j,\quad \bz_k^{**}:= r_m \bz_k,\nonumber
\end{align}
and thus the corresponding scaling result
\begin{align}
    \lambda\Big(B^m(\bz_j^{**},\Vert \bz_j^{**} \Vert)\cup(B^m(\bz_k^{**},\Vert \bz_k^{**} \Vert)\Big) = (r_m)^m \lambda\Big(B^m(\bz_j,\Vert \bz_j \Vert)\cup(B^m(\bz_k,\Vert \bz_k \Vert)\Big).\nonumber
\end{align}
Setting $r_m := \Big((n-2)(1+\epsilon)^2 g(\psi(\bx))\Big)^{1/m}$ and noting that the corresponding Jacobi matrix satisfies $\vert J\vert = (n-2)(1+\epsilon)^2 g(\psi(\bx))$, we have
\begin{align}
    J_{n,1}\geq \frac{(1-\epsilon)n^2}{(1+\epsilon)^4(n-2)^2}  \iint_{\Gamma_{m,l_m\delta}}\!\!\exp\Big(-\lambda\big(B^m(\bz_j^{**},\Vert \bz_j^{**} \Vert)\cup(B^m(\bz_k^{**},\Vert \bz_k^{**} \Vert)\big)\Big)\ds\lambda(\bz_j^{**})\ds\lambda(\bz_k^{**}).\nonumber
\end{align}
Following the same procedure as we discussed before, as $n$ goes to infinity, we then obtain the same lower bound.

Matching the upper and lower bounds then yields
\begin{align}
    \lim_{n \to \infty}\mathrm{E}M(\X_i)=&\lim_{n \to \infty} n^2\mathrm{P}(Q_j \cap Q_k|\X_i=\bx)
    =\lim_{n \to \infty} I_{n,2}= \ko_m.\nonumber
\end{align}

Finally,
\begin{align}
     \lim_{n \to \infty}\frac{\E M_{total}}{n}=\lim_{n \to \infty}\mathrm{E}\sum_{i=1}^n \frac{M(\X_i)}{n}=\lim_{n \to \infty}\sum_{i=1}^n \frac{\mathrm{E}M(\X_i)}{n}=\ko_{m}.\nonumber
\end{align}

This completes the proof.
\end{proof}

\section{Proofs of the rest results}
\subsection{Proof of Lemma \ref{cons}(a)}
\begin{proof}[Proof of Lemma \ref{cons}(a)]
Recall the expression for $\kq_m$:
\begin{align*}
    \kq_m:=\Big\{2-I_{3/4}\Big(\frac{m+1}{2},\frac{1}{2}\Big)\Big\}^{-1},
~~~~
I_{x}(a,b):=\frac{\int_{0}^{x}t^{a-1}(1-t)^{b-1} \ds t}
                 {\int_{0}^{1}t^{a-1}(1-t)^{b-1} \ds t}.
\end{align*}
To prove the monotone of $\kq_m$ as the dimension $m$ increases, it suffices to show that $I_{x}(a,b)$ decreases as argument $a$ increases when $x\in[0,1]$ and $a,b>0$. For some $\epsilon>0$, we directly compare the values of $I_{x}(a,b)$ and $I_{x}(a+\epsilon,b)$. We have
\begin{align*}
    &\big(I_{x}(a,b)\big)^{-1} = 1+ \frac{\int_{x}^{1}t^{a-1}(1-t)^{b-1} \ds t}
                 {\int_{0}^{x}t^{a-1}(1-t)^{b-1} \ds t},\\
    &\big(I_{x}(a+\epsilon,b)\big)^{-1} = 1+ \frac{\int_{x}^{1}t^{a-1}t^{\epsilon}(1-t)^{b-1} \ds t}
                 {\int_{0}^{x}t^{a-1}t^{\epsilon}(1-t)^{b-1} \ds t}.
\end{align*}
Observe that
\begin{align*}
    \int_{x}^{1}t^{a-1}t^{\epsilon}(1-t)^{b-1} \ds t > x^{\epsilon}\int_{x}^{1}t^{a-1}(1-t)^{b-1} \ds t,\\
    \int_{0}^{x}t^{a-1}t^{\epsilon}(1-t)^{b-1} \ds t < x^{\epsilon}\int_{0}^{x}t^{a-1}(1-t)^{b-1} \ds t,
\end{align*}
which implies that $\big(I_{x}(a,b)\big)^{-1} < \big(I_{x}(a+\epsilon,b)\big)^{-1}$. Equivalently, $I_{x}(a+\epsilon,b) < I_{x}(a,b)$ for every $\epsilon > 0$.

With monotony, we have $\kq_m \in (\kq_{\infty},\kq_1]$, where $\kq_{\infty} := \lim_{m\rightarrow\infty}\kq_m$. Here, $\kq_1 = 2/3$, and for $\kq_{\infty}$, notice that for some sufficient small $\delta > 0$,
\begin{align*}
    \frac{\int_{0}^{x}t^{a-1}(1-t)^{b-1} \ds t}
                 {\int_{x}^{1}t^{a-1}(1-t)^{b-1} \ds t}
    \leq \frac{\int_{0}^{x}t^{a-1}(1-t)^{b-1} \ds t}
                 {\int_{x+\delta}^{1-\delta}t^{a-1}(1-t)^{b-1} \ds t}
    \leq \frac{x^a}{(1-x-2\delta)(x+\delta)^{a-1}\delta^{b-1}}.
\end{align*}
Thus, $\lim_{a\rightarrow\rightarrow}I_{x}(a,b)=0$, and hence $\kq_{\infty}= 1/2$.

In conclusion, $\kq_m\in(1/2,2/3]$ is strictly decreasing as $m$ increases.
\end{proof}

\subsection{Proof of Lemma \ref{cons}(b)}
\begin{proof}[Proof of Lemma \ref{cons}(b)]
Recall the expression for $\ko_m$:
\begin{align*}
    &\ko_{m}:=\iint_{\Gamma_{m;2}}\exp\Big[-\lambda\Big\{B(\mw_1,\lVert\mw_1\rVert_{})\cup B(\mw_2,\lVert\mw_2\rVert_{})\Big\}\Big]\ds(\mw_1,\mw_2),\\
    &\Gamma_{m;2}:=\Big\{(\mw_1,\mw_2)\in(\R^m)^2: \max(\lVert\mw_1\rVert_{},\lVert\mw_2\rVert_{})<\lVert\mw_1-\mw_2\rVert_{}\Big\}.
\end{align*}
Considering the symmetry of $\mw_1$ and $\mw_2$ is the integral, we have
\begin{align*}
    &\ko_{m}=2\iint_{\Gamma_{m;2}^{*}}\exp\Big[-\lambda\Big\{B(\mw_1,\lVert\mw_1\rVert_{})\cup B(\mw_2,\lVert\mw_2\rVert_{})\Big\}\Big]\ds(\mw_1,\mw_2),\\
    \text{where } &\Gamma_{m;2}^{*}:=\Gamma_{m;2}\cap\Big\{(\mw_1,\mw_2)\in(\R^m)^2: \lVert\mw_1\rVert_{}>\lVert\mw_2\rVert_{}\Big\}
    .
\end{align*}
We first prove that for any positive integer $m$, $\ko_m<2$. The expression for $\ko_m$ means that
\begin{align*}
    \ko_{m}<&2\iint_{\Gamma_{m;2}^{*}}\exp\Big[-\lambda\Big\{B(\mw_1,\lVert\mw_1\rVert_{})\Big\}\Big]\ds\mw_2\ds\mw_1\\
    =&2\int_{\R^m}\exp\Big[-\lambda\Big\{B(\mw_1,\lVert\mw_1\rVert_{})\Big\}\Big]V_{m}\lVert\mw_1\rVert_{}^m\ds\mw_1\\
    =&2\int_{0}^{\infty}\exp(-V_{m}t^{m})V_{m}t^m\cdot\big(m V_m t^{m-1}\ds t\big)\\
    =&\int_{0}^{\infty}\exp(-V_{m}t^{m})\ds \big(V_{m}^{2}t^{2m}\big)=2,
\end{align*}
in which we apply polar coordinates transformation and denote $t = \lVert\mw_1\rVert_{}$ in the transformation.

Then we consider the limit behavior and prove that $\limsup_m\ko_m\leq 1$.
The definition of $\Gamma_{m;2}^{*}$ shows that, for $(\mw_1,\mw_2)\in \Gamma_{m;2}^{*}$, we have $\mw_1\notin B(\mw_2,\lVert\mw_2\rVert_{})$ and $\lVert\mw_2\rVert_{}<\lVert\mw_1\rVert_{}$. Thus, for fixed $\mw_1$, the Lebesgue measure of $B(\mw_1,\lVert\mw_1\rVert_{})\cap B(\mw_2,\lVert\mw_2\rVert_{})$ can be bounded by the restrictions above, that is,
\begin{align*}
    \lambda\Big\{B(\mw_1,\lVert\mw_1\rVert_{})\cap B(\mw_2,\lVert\mw_2\rVert_{})\Big\}<&\lambda\Big\{B(\mw_1,\lVert\mw_1\rVert_{})\cap B(\textbf{0},\lVert\mw_1\rVert_{})\Big\}=\big(2-\frac{2}{\kq_m}\big)\lambda\Big\{B(\mw_1,\lVert\mw_1\rVert_{})\Big\}.
\end{align*}
We denote $\epsilon_m = 2-2/\kq_m$. According to Lemma \ref{cons}(a), it is known that $\lim_{m\rightarrow\infty}\epsilon_m = 0$.
Applying the estimation, we have the bound
\begin{align*}
    &\lambda\Big\{B(\mw_1,\lVert\mw_1\rVert_{})\cup B(\mw_2,\lVert\mw_2\rVert_{})\Big\}\\
    =&\lambda\Big\{B(\mw_1,\lVert\mw_1\rVert_{})\Big\}+\lambda\Big\{B(\mw_2,\lVert\mw_2\rVert_{})\Big\}-\lambda\Big\{B(\mw_1,\lVert\mw_1\rVert_{})\cap B(\mw_2,\lVert\mw_2\rVert_{})\Big\}\\
    >&(1-\epsilon_m)V_m \lVert\mw_1\rVert_{}^m +V_m \lVert\mw_2\rVert_{}^m.
\end{align*}
We then get back to the expression for $\ko_m$:
\begin{align*}
    \ko_{m}<&2\iint_{\Gamma_{m;2}^{*}}\exp\Big[-(1-\epsilon_m)V_m \lVert\mw_1\rVert_{}^m -V_m \lVert\mw_2\rVert_{}^m\Big]\ds(\mw_1,\mw_2)\\
    =&2\iint_{\Gamma_{m;2}^{*}}\Big(\exp\Big[-V_m \lVert\mw_2\rVert_{}^m\Big]\ds\mw_2\Big)\exp\Big[-(1-\epsilon_m)V_m \lVert\mw_1\rVert_{}^m\Big]\ds\mw_1\\
    <&2\int_{\R^m}\Big(\int_{\mw_2:\lVert\mw_2\rVert_{}<\lVert\mw_1\rVert_{}}\exp\Big[-V_m \lVert\mw_2\rVert_{}^m\Big]\ds\mw_2\Big)\exp\Big[-(1-\epsilon_m)V_m \lVert\mw_1\rVert_{}^m\Big]\ds\mw_1\\
    =&2\int_{\R^m}\Big(1-e^{-V_m \lVert\mw_1\rVert_{}^m} \Big)\exp\Big[-(1-\epsilon_m)V_m \lVert\mw_1\rVert_{}^m\Big]\ds\mw_1\\
    =&2\Big(\frac{1}{1-\epsilon_m}-\frac{1}{2\sqrt{1-\epsilon_m}}\Big)\rightarrow 1.
\end{align*}
The calculation follows from the polar coordinates transformation as used before. In conclusion, we have $\limsup_m\ko_m\leq 1$. Combining this with $\ko_m < 2$, we obtain $\sup_{m}\ko_m <2$.
\end{proof}

\subsection{Proof of Lemma \ref{U lemma}}
\begin{proof}[Proof of Lemma \ref{U lemma}]
We consider approximating $\mu(U^d(\bx_i,\bx_j))$ as
\begin{align*}
    \mu(U^d(\bx_i,\bx_j))=&
    \psi_{*}\mu(\psi(U^d(\bx_i,\bx_j)\cap\mathcal{M}))=\int_{\psi(U^d(\bx_i,\bx_j)\cap\mathcal{M})}g\ds\lambda(\bx_j).
\end{align*}
Using Lebesgue Differentiation Theorem (LDT), there exists $\delta_1>0$ such that for $\mu$-a.e. $\bx_i$ and every $\bx_j$ satisfying $\Vert \bx_j-\bx_i\Vert < \delta_1$, we have
\begin{align}\label{direct}
    \frac{\mu(U^d(\bx_i,\bx_j))}{g(\psi(\bx_i))\lambda(\psi(U^d(\bx_i,\bx_j)\cap\mathcal{M}))}\in \Big[1-\frac{\epsilon}{3},1+\frac{\epsilon}{3}\Big].
\end{align}
For every $\alpha > 0$, we first define 
\begin{align*}
    U^m_{\alpha}(\bx,\by) := B^m\big(\bx,\alpha\Vert \bx-\by\Vert\big)\cup B^m\big(\by,\alpha\Vert \bx-\by\Vert\big).
\end{align*}
Taking a specific form of $\psi$, for example, orthogonal projection onto the tangent plane $\pi$, we have for every $\gamma >0$ and $\bx_j$ satisfying $\Vert \bx_j-\bx_i\Vert < \delta_2$, we have
\begin{align}
    U^m_{1-\gamma}\big(\psi(\bx_i),\psi(\bx_j)\big)
    \subset \psi\big(U^d(\bx_i,\bx_j)\cap\m\big)
    \subset U^m_{1-\gamma}\big(\psi(\bx_i),\psi(\bx_j)\big).\label{shuyu}
\end{align}
In fact, considering the properties of orthogonal projection, we have for every $\bx\in U^d(\bx_i,\bx_j)\cap\mathcal{M}$,
\begin{align*}
    \Vert\psi(\bx)-\psi(\bx_i)\Vert \leq \Vert \bx-\bx_i \Vert\leq\Vert \bx_i-\bx_j \Vert.
\end{align*}
\begin{lemma}\label{slope}
(i) For any $\bx \in \m$ and any $\alpha > 0$, there exists $\delta > 0$ such that for every $\bx_i \in \m$ satisfying $\Vert \bx_i-\bx \Vert{} < \delta$, we have that the angle between the vector $\bx_i-\bx$ and its projection onto $\pi$, which is the tangent plane of $\m$ at point $\bx$, is less than $\alpha$. 

(ii) Furthermore, there exists $\delta^{*}>0$ such that for every $\bx_i,\bx_j\in\m$ satisfying $\max\{\Vert \bx_i-\bx \Vert{}, \Vert \bx_j-\bx \Vert{}\}< \delta^{*}$, we have that the angle between the vector $\bx_i-\bx_j$ and its projection onto the tangent plane $\pi$ is less than $\alpha$.
\end{lemma}

Applying Lemma \ref{slope}, we can select a sufficiently small $\delta_2$ such that
\begin{align*}
    \Vert \bx_i-\bx_j \Vert \leq (1+\gamma)\Vert \psi(\bx_i)-\psi(\bx_j) \Vert.
\end{align*}
This reveals that for every 
\[
\bx\in B\Big(\bx_i,\Vert \bx_i-\bx_j\Vert\Big)\bigcap\m, 
\]
it holds true that 
\[
\psi(\bx)\in B\Big(\psi(\bx_i),(1+\gamma)\Vert \psi(\bx_i)-\psi(\bx_j)\Vert\Big).
\]
The above process can be repeated for the pair $(\bx, \bx_j)$ similarly. Putting them together, we get the second ``$\subset$'' of Equation \eqref{shuyu}:
\begin{align*}
    \psi\Big(U^d(\bx_i,\bx_j)\bigcap\mathcal{M}\Big)\subset U^m_{1+\gamma}\big(\psi(\bx_i),\psi(\bx_j)\big).
\end{align*}
For the first ``$\subset$'', the proof is similar. In detail, for every 
\[
\bz\in B\Big(\psi(\bx_i),(1-\gamma)\Vert \psi(\bx_i)-\psi(\bx_j)\Vert\Big)\bigcap V, 
\]
we can select a sufficiently small $\delta_2$ such that
\begin{align*}
    \Vert \psi^{-1}(\bz)-\bx_i\Vert \leq \Big(1+\frac{\gamma}{2}\Big)\Vert \bz-\psi(\bx_i)\Vert\leq\Big(1-\frac{\gamma}{2}\Big)\Vert \psi(\bx_i)-\psi(\bx_j)\Vert\leq \Vert \bx_i-\bx_j\Vert,
\end{align*}
which implies that
\begin{align*}
    \psi^{-1}\Big(U^m_{1-\gamma}\big(\psi(\bx_i),\psi(\bx_j)\big)\Big)\subset U^d(\bx_i,\bx_j)\bigcap\mathcal{M}.
\end{align*}
Applying the mapping $\psi$ to both sides and repeating the process for $\bx$ and $\bx_j$, we obtain the first ``$\subset$'' of Equation \eqref{shuyu}:
\begin{align*}
    U^m_{1-\gamma}\big(\psi(\bx_i),\psi(\bx_j)\big)\subset\psi(U^d(\bx_i,\bx_j)\cap\mathcal{M}).
\end{align*}

With Equation \eqref{shuyu} and noticing that
\begin{align*}
    \lim_{\Vert \bx_j-\bx_i\Vert\rightarrow0}\frac{\Vert \psi(\bx_i)-\psi(\bx_j)\Vert}{\Vert \bx_i-\bx_j\Vert} = 1,
\end{align*}
one can directly obtain that
\begin{align}\label{contain}
    \lim_{\Vert \bx_j-\bx_i\Vert\rightarrow0}\frac{\lambda\big(\psi(U^d(\bx_i,\bx_j)\cap\mathcal{M})\big)}{\lambda\big(U^m(\bx_i,\bx_j)\big)}=\lim_{\Vert \bx_j-\bx_i\Vert\rightarrow0}\frac{\lambda\big(\psi(U^d(\bx_i,\bx_j)\cap\mathcal{M})\big)}{\lambda\big(U^m(\psi(\bx_i),\psi(\bx_j))\big)}=1.
\end{align}
Combining \eqref{direct} and \eqref{contain} completes the proof of Lemma \ref{U lemma}.
\end{proof}

\subsection{Proofs of Lemma \ref{B lemma}, \ref{B2 lemma}, \ref{B2cup lemma} and \ref{area lemma}}
\begin{proof}[Proofs of Lemma \ref{B lemma}, \ref{B2 lemma} and \ref{B2cup lemma}] All these lemmas share a similar proof to Lemma \ref{U lemma} that we have proved above. 
Details are hence omitted.
\end{proof}
\begin{proof}[Proof of Lemma \ref{area lemma}] We consider the pushforward $\psi^{-1}_{*}\lambda$ and Lebesgue measure $\lambda$ instead of $\mu$ and the pushforward $\psi_{*}\mu$ in this lemma. Hence, we just set $g\equiv 1$ and its proof can be regarded as a special case of the proof sketch of Lemma \ref{B2 lemma}.
\end{proof}

\subsubsection{Proof of Lemma \ref{slope}}
\begin{proof}[Proof of Lemma \ref{slope}] We denote by $\pi$ the tangent plane of $\m$ at point $\bx$ and $\pi^{\perp} $ the orthogonal complement of $\pi$ in $\R^d$. Moreover, we denote by $P_V$, $Q_V$ the orthogonal projection operators on $\pi$ and $\pi^{\perp}$, respectively.  The orthogonal decomposition of a vector $\bz$ with respect to $\pi$ can then be shown as $\bz = P_V(\bz) + Q_V(\bz)$. 

Define $\theta(\bz)$ as the angle between $\bz$ and its projection onto the tangent plane $\pi$. When $\bz=\bm{0}$, we define $\theta(\bz) = 0$. Using the notation above, one has
\begin{align*}
    \tan(\theta(\bx_i-\bx))= \frac{\Vert Q_V(\bx_i-\bx) \Vert{}}{\Vert P_V(\bx_i-\bx) \Vert{}}, ~~~~ \bx_i\neq \bx.
\end{align*}
We now consider $\theta(\bx_i-\bx)$ as a function of $\bx_i$. Since $\m$ is a smooth manifold, $Q_{V}(\bx_i-\bx)$ and $P_{V}(\bx_i-\bx)$ are both continuous functions. With the definition of the tangent plane $\pi$, we know that $\theta(\bx_i-\bx)$ is continuous at point $\bx_i=\bx$. Thus $\theta(\bx_i-\bx)$ is a continuous function on $\m$. We can select a fixed $\delta_1 >0$ so that $\theta(\bx_i-\bx)$ is uniformly continuous on $\m\cap B^d(\bx,\delta_1)$. Since $\theta(\bx-\bx)=0$, there exists  $0<\delta<\delta_1$ such that for any $\bx_i \in \m$ and $\Vert \bx_i-\bx \Vert{}<\delta$, we have $\theta(\bx_i-\bx) < \alpha$, which is the first claim.

For the second claim, we just have to modify $\theta(\bz)$ as follows: we define $\theta(\bx_i,\bx_j)$ to be the angle between the vector $\bx_i-\bx_j$ and its projection onto the tangent plane $\pi$ so that
\begin{align*}
    \tan(\theta(\bx_i,\bx_j)) = \frac{\Vert Q_V(\bx_i-\bx_j) \Vert}{\Vert P_V(\bx_i-\bx_j) \Vert},
    ~~~~ \bx_i\neq \bx_j.
\end{align*}
Additionally, define $\theta(\bx_i,\bx_j) = 0$ when $\bx_i=\bx_j$. We select $\delta_1^{*}$ to make sure that $\theta(\bx_i,\bx_j)$ exists. The second claim can then be obtained similarly by following the proof above.
\end{proof}

\subsubsection{Proof of Lemma \ref{newLDT}}
\begin{proof}[Proof of Lemma \ref{newLDT}]
It is a simple corollary to Lebesgue Differential Theorem, which can be derived by applying LDT to $f$ and $g$ respectively:
\begin{align*}
    \lim_{U\rightarrow \bx, U\in\mathcal{V}}\frac{1}{\vert U\vert}\int_{U}f g\ds\lambda
    =f(\bx)g(\bx)\lim_{U\rightarrow \bx, U\in\mathcal{V}}\frac{1}{\vert U\vert}\int_{U}\ds\lambda
    =f(\bx)\lim_{U\rightarrow \bx, U\in\mathcal{V}}\frac{1}{\vert U\vert}\int_{U} g\ds\lambda.
\end{align*}
This completes the proof.
\end{proof}

\bibliographystyle{chicago}
\bibliography{ref}

\end{document}